\documentclass[11pt]{amsart}
\usepackage{a4wide}
\usepackage{amsmath}
\usepackage{graphicx}
\usepackage{amsfonts}
\usepackage{amssymb}
\usepackage{MnSymbol}
\usepackage{esint}
\usepackage{mathtools}
\usepackage[font=small,format=hang,labelfont={sf,bf}]{caption}
\usepackage{color}
\usepackage{mathrsfs}
\usepackage{epsfig}
\usepackage{subfig}
\usepackage{url}
\theoremstyle{plain}
\newtheorem{theorem}{Theorem}[section]

\newtheorem{lemma}[theorem]{Lemma}
\newtheorem{proposition}[theorem]{Proposition}
\newtheorem{corollary}[theorem]{Corollary}
\newtheorem{remark}[theorem]{Remark}
\newtheorem{definition}[theorem]{Definition}
\theoremstyle{definition}
\theoremstyle{remark}
\numberwithin{equation}{section}

\newcommand{\eps}{\varepsilon}

\newcommand{\diam}{\mathrm{diam}\,}

\newcommand{\R}{\mathbb{R}}
\newcommand{\N}{\mathbb{N}}
\newcommand{\Z}{\mathbb{Z}}

\newcommand{\El}{\mathcal{E}_{el}}
\newcommand{\Ec}{\mathcal{E}_{core}}
\newcommand{\F}{\mathcal{F}}

\newcommand{\dist}{\mathrm{dist}}

\newcommand{\supp}{\mathrm{supp}\,}

\newcommand{\Curl}{\mathrm{Curl}\,}

\newcommand{\Div}{\mathrm{Div}}

\newcommand{\weakly}{\rightharpoonup}           
\newcommand{\weakstar}{\stackrel{*}{\weakly}}   

\newcommand{\loc}{\mathrm{loc}}

\newcommand{\curl}{\mathrm{Curl}}

\newcommand{\Ha}{\mathcal{H}}

\newcommand{\LL}{\mathbin{\vrule height 1.6ex depth 0pt width
0.13ex\vrule height 0.13ex depth 0pt width 1.3ex}}

\newcommand{\mean}{-\!\!\!\!\!\!\int}

\newcommand{\harm}{\textup{harm}}

\newcommand{\res}{\mathop{\hbox{\vrule height 7pt width .5pt depth 0pt
\vrule height .5pt width 6pt depth 0pt}}\nolimits}
\usepackage{latexsym}

\title[Grain boundary energy] 
{On the Read-Shockley energy for grain boundaries in poly-crystals}
\author{Martino Fortuna}
\email{martino.fortuna@uniroma1.it}
\author{Adriana Garroni}
\email{adriana.garroni@uniroma1.it}
\author{Emanuele Spadaro}
\email{emanuele.spadaro@uniroma1.it}

\begin{document}

\begin{abstract}
	In the 50's Read and Shockley  proposed a formula for the energy of small angle grain boundaries in polycrystals based on linearised elasticity and an ansazt on the distribution of incompatibilities of the lattice at the interface. In this paper we derive a sharp interface limiting functional starting from a nonlinear semidiscrete model for dislocations proposed by Lauteri--Luckhaus in \cite{LL2}. Building upon their analysis we obtain, via $\Gamma$-convergence, an interfacial energy depending on the rotations of the grains and the relative orientation of the interface which agrees for small angle grain boundaries with the Read and Shockley logarithmic scaling.
\end{abstract}

\maketitle

\vspace{1cm}

\begin{center}
\begin{minipage}{12cm}
\small{
\tableofcontents
}
\end{minipage}
\end{center}

\vspace{.5cm}

\section{Introduction}\label{intro}



In \cite{LL2} Lauteri and Luckhaus undertake the analysis of a very general geometric functional for the description of dislocations' configurations in metal plasticity.
In fact, metals are actually crystals and their mechanic properties are strongly effected by the underlying ordered disposition of atoms. Precisely, metals exhibit a local order in regions (called \emph{grains}) where the underlying structure is closed to a rotated reference lattice. These grains are separated by boundaries and the incompatibility of the crystal lattice along these boundaries induces an elastic distortion, which in certain  scales can be modeled with an energies concentrated on these boundaries. The understanding of the dependence of this \emph{surface tension} upon the relative orientations of the grains is crucial in order to clarify the mechanical properties of poly-crystalline materials.
It is understood that the incompatibilities of the crystal lattices at the interface are resolved by means of the presence of defects, such as dislocations, whose density depends on the discrepancy at the interface.

Dislocations are, roughly speaking, the effect of locally lattice invariant slips along crystallographic planes and at the equilibrium may be identified with the lines separating the regions undergoing different slips in such crystallographic planes. Outside a neighborhood of these one-dimensional objects, at a distance comparable with the lattice spacing (the so called \emph{core regions} of the dislocations), the material is a local distortion of a perfect lattice.
For this reason, the typical models for dislocations are somehow hybrid: they show a cohexistence of continuum variables for the bulk elastic deformation and a microscopic parameter $\eps>0$ for the size of the core region of the dislocations (proportional to the lattice spacing, essentially the range where the continuum approximation fails). These kind of models are often referred to as \emph{semi-discrete models} and their validity has been confirmed by the asymptotic analysis of fully discrete models  (see, e.g, \cite{Ponsiglione,Alicandro-DeLuca-Garroni-Ponsiglione, DeLuca-Garroni-Ponsiglione,Luckhaus-Mugnai}).

\medskip

The model proposed by Lauteri and Luckhaus integrates perfectly into this framework.
They consider the problem in a two-dimensional section of a crystal $\Omega\subset \R^2$  and the state space is described by couples $(A,S)$, where $A:\Omega\to \R^{2\times 2}$ is a matrix fields and $S\subset\Omega$ is a closed set.
Here $S$ is assumed to contain the defects of the materials, whereas $A$ outside of $S$ represents the local elastic deformation of the material, and hence $A$ is locally a gradient in $\Omega\setminus S$.
In the presence of defects, $A$ is not globally a gradient and its circulation on closed loops around the connected components of $S$ should be a vector of the underlying lattice, the so called \emph{Burgers vectors}.

The energy introduced by Lauteri and Luckhaus has two components: a core energy proportional to the area of the core region (the tubular neighborhood of $S$ at distance $\lambda \eps$, with $\lambda\geq 1$ a fixed constant and $\eps>0$ representing the lattice spacing)
$$
|B_{\lambda\eps}(S)|,
$$
and an elastic energy outside the core region of the form
$$
\int_{\Omega\setminus  B_{\lambda\eps}(S)} W(A)\, dx,
$$
with $W$ a suitable energy density.

Clearly, since $S$ contains the incompatibilities of the elastic deformation, we have that $A$ is $\Curl$-free away from $S$, and moreover $A$ and $S$ are further linked by the extra-condition related to the quantization of the circulations in terms of the Burgers vectors.

\medskip
The main objective of the present paper is the derivation of a sharp interface energy for grain boundaries in dimension two starting from the Lauteri--Luckhaus' model. Roughly speaking, the energy is defined for fields which are locally constant rotations.
Given two grains $E_i,E_j\subseteq \Omega$ and the corresponding constant matrix fields $R_i,R_j\in SO(2)$,  representing the rotations of the underlying crystal structure with respect to a reference lattice, we would like to measure the  line tension of the boundary between the grains as
$$
\Phi(R_i,R_j,\nu_{i,j})\Ha^1(\Gamma_{i,j}),
$$
where  $\Gamma_{i,j}=\partial E_i\cap \partial E_j$ and $\Phi$ is the energy per unit length.

\medskip

In the last decades dislocations have been treated in mathematical terms giving rise to a large literature that deals with microscopic, mesoscopic and macroscopic scales (see .... and the references therein).
The paper of Lauteri--Luckhaus \cite{LL2} is a milestone in this theory, because for the first time the energetic regime which leads to the formation of grain boundaries is investigated without ad hoc assumptions on the geometric distribution of defects.
In \cite{LL2} the authors consider the problem of interpolating between two small symmetric rotations of angle $\pm \alpha$ imposed on the lateral sides of a square and they prove three kinds of results:
\begin{itemize}
	\item[1.] an \emph{upper bound}, consisting in the construction of a configuration matching the right boundary conditions in agreement with the \emph{Read and Shockley ansatz} on the energy of the interface (see \cite{Read-Shockley,Howe}):
	$$
	\textup{energy per unit length} \sim \eps\; \alpha\; |\log\alpha|, \qquad 0<\alpha<<1;
	$$
	\item[2.] the convergence of suitable sequences of configurations with finite energy to \emph{micro-rotations} (see Definition \ref{def-micro-rotation} below), modelling the grains with constant oriented lattices;
	\item[3.] a \emph{lower bound} of the energy matching (up to a constant) the \emph{Read and Shockley ansatz} for grain boundary between small rotations.
\end{itemize}

One of the main contributions of \cite{LL2} is the introduction of the variable $S$, which adds flexibility to the model and allows to treat scalings at which geometric necessary dislocations emerge. In doing this, the authors introduce and develop a wealth of innovative ideas and tools, beyond the existing literature on related problems with topological singularities.
It is also worth noticing that the functional proposed by Lauteri and Luckhaus has an essential geometric structure: potentially it is a lower bound that can be used as a comparison energy for several other interesting models, as optimal partition problems, image segmentations etc\ldots.

\medskip

In this paper we build upon the analysis in \cite{LL2} and complete the picture for the asymptotics of the Lauteri--Luckhaus energies, proving a full $\Gamma$-convergence result to a limiting energy of the form
\begin{equation}\label{e.gamma limite}
	\int_{J_A} \Phi \big(A^+,A^-,\nu_A\big)\, d\Ha^1,
\end{equation}
for $A$ a $SBV$-field with values in $SO(2)$ and whose derivative has only jump part ($J_A$ is the jump set,  $\nu_A$ its normal vector, and $A^+,A^-$ the traces from both sides of the jump).
In particular, we extend the results in \cite{LL2} in four directions:
\begin{itemize}
	\item[1.] we prove a \emph{compactness} result in the natural topology for sequences of fields with bounded energies (whereas the compactness of Lauteri--Luckhaus was valid only for distinguished competitors);
	\item[2.] we \emph{extend the model} incorporating  the anisotropy of the underlying crystalline structure;
	\item[3.] we find an \emph{asymptotic formula} for the energy density $\Phi$, showing that the lower and upper bound in \cite{LL2} can be upgraded into an interfacial energy;
	\item[4.] we find an \emph{improved construction for the upper bound} which depends on the relative orientations of the grains (not necessarily symmetric with respect to the interface), matching the Read--Shockley logarithmic scaling. Moreover, we are able to treat the case of interfaces between two incompatible lattices as studied by Van der Merwe \cite{VdM}.
\end{itemize}

In order to show such results, we need to use the entire machinary introduced and developed by Lauteri and Luckhaus, most notably their \emph{logaritmic estimate} (see Theorem \ref{t.LL} below) and the set of ideas in the modification of the matrix field $A$ in order to achieve good bounds and compactness properties.
Concerning this last point, for our asymptotic analysis it is necessary to revisit some of the constructions done in \cite{LL2}, although we do not introduce any essentially new ideas.

The proof of the $\Gamma$-convergence, passes through an appropriate cell problem formula and the analysis of the thermodynamic limit for the rescaled functionals.  A crucial step in this strategy is to show that the thermodynamic limit actually gives the optimal energy per unit length of a straight interface. In doing this we prove that optimal sequences concentrates around the interface and exhibit a sort of equi-distribution of defects (in the spirit of \cite{Alberti-Choksi-Otto, Spadaro} in agreement with the ideas behind the explicit upper bound  constructions and the computations in \cite{Read-Shockley,VdM,DeLuca-Ponsiglione-Spadaro}).


\section{Setting and main result}

\subsection{Parameters and classes of competitors}
In the sequel $\Omega\subset\R^2$ is a  bounded open set which represents a (possibly defected) material body, and $\lambda,\tau>0$ are two fixed parameters of the energy functionals.
For every $\eps>0$ we define the class of admissible pairs $(A,S)\in\mathcal{C}_\eps(\Omega)$, where $A\in L^1(\Omega,\R^{2\times 2})$ and $S\subset \Omega$ is a relatively closed set satisfying the following conditions:
\begin{itemize}
	\item[(H1)] $S\supseteq \supp(\Curl A)$;
	\item[(H2)] 
	for every simple rectifiable closed curve
	$\gamma$ in $\Omega\setminus S$
	one has that
	\begin{align*}
		\int_{\gamma} A
		\in  \tau\eps \Z^2.
	\end{align*}
\end{itemize}

For a continuous field $A$ and a rectifiable curve $\gamma:[0,1]\to \Omega$ the line integral $\int_\gamma A$ is
\[
\int_\gamma A := \int_0^1 A(\gamma(t))\,\dot\gamma(t) \, dt,
\]
whereas for $A\in L^1(\Omega;\R^{2\times 2})$ one should understand it in the weak sense:
\begin{equation}\label{eq-def-circuitazione}
	\int_\gamma A: = -\textup{sgn}(\gamma)\int_\Omega A(x)\, \nabla^\perp \varphi(x) \, dx,
\end{equation}
with $\varphi\in C_c^\infty(\Omega)$ any function such that
\[
\varphi=1\quad \textup{on}\quad D\qquad \textup{and}\qquad \big({\rm supp }\,\varphi\big)\setminus D\subset \Omega\setminus S,
\]
where $D\subset \Omega$ is the open set such that $\gamma=\partial D$ and $\textup{sgn}(\gamma) = \pm 1$ according to the orientation of $\gamma$ (anticlockwise $+1$, clokwise $-1$).
Recall that
\[
\nabla^\perp \varphi := J\nabla \varphi,\qquad
J:= \begin{pmatrix}
	0 & 1 \\
	-1 & 0
\end{pmatrix}.
\]
This definition reduces to the classical circulation if $A$ is smooth, and it is well posed, since the right hand side of \eqref{eq-def-circuitazione} does not depend on the precise choice of $\varphi$: indeed, if we take a second admissible test $\psi\in C^\infty(\Omega)$ with $\psi=1$ on $D$, and
$$({\rm supp}\,\psi)\setminus D\subseteq \Omega\setminus S,$$
then $\varphi-\psi\in C_c^\infty(\Omega)$ and $${\rm supp}(\varphi-\psi)\subseteq \Omega\setminus \supp(\Curl A),$$ and hence $\int A\,\nabla^\perp(\varphi-\psi)= 0$.

	In this model $A$ represents a field that locally (outside the incompatibilities contained in the set $S$) is a gradient of a function that maps the material body $\Omega$ in a reference configuration which is given by a perfect crystal, here the square lattice $\eps\tau\Z^2$. This property is expressed by the circulation condition (H2). A different Bravais lattice (e.g., triangular lattice) can be also considered without changing our analysis.

	\begin{remark}\label{rem-invariance-rotation0}
	 In view of the above interpretation for $A$, we remark that a change of variable in the configuration $\Omega$ should not change the circulation. Therefore, changing the domain $\Omega$ with a rotation  $R\Omega$, we must change the field accordingly. Namelly,  $A_R: R\Omega\to\R^{2\times2}$ given by $A_R(y)=A(R^{-1}y)R^{-1}$ (which is clearly what happens if $A=\nabla u$ and $A_R=\nabla u_R$ with $u_R(y)=u(R^{-1}y)$). Condition (H2) does not change, indeed
	 	\begin{align*}
	 	\int_{\gamma} A
	 =\int_{R\gamma} A_R.
	 \end{align*}

\end{remark}

\subsection{The energy functional}
For any  $(A,S)\in \mathcal{C}_\eps(\Omega)$ and for any open subset $\Omega'\subset \Omega$, we define
\begin{itemize}
	\item the {\bf elastic energy}
	$$
	\El^\eps(A,S;\Omega'):= \int_{\Omega'\setminus B_{\lambda\eps}(S)} \dist^2(A(x), SO(2))\, dx,
	$$
	\item the {\bf core energy}
	$$
	\Ec^\eps(S;\Omega'):=|B_{\lambda\eps}(S)\cap \Omega'|\,,
	$$
\end{itemize}
where $|E|$ is the $2$-dimensional Lebesgue measure and $B_{\lambda\eps}(E)$ is the $\lambda\eps$-neighborhood in $\Omega$ of any subset $E\subset\Omega$,
$$
B_{\lambda\eps}(E):=\Big\{x\in \Omega : \ \dist(x, E)< \lambda\eps \Big\}.
$$
The energy functional introduced by Lauteri--Luckhaus in \cite{LL2} is given by the sum of the two terms:
\begin{equation}\label{energy}
	\F_\eps(A,S;\Omega'):=
	\El^\eps(A,S;\Omega')+\Ec^\eps(S;\Omega').
\end{equation}
The energy of a matrix valued field $A\in L^1(\Omega;\R^{2\times 2})$ is then defined by optimazing in the set $S\supset \supp(\Curl A)$ and rescaling in $\eps$ the Lauteri--Luckhaus energy:
$$
F_\eps(A;\Omega'):=
\frac{1}{\eps}\inf\Big\{\F_\eps(A,S;\Omega')\,:\,
(A,S)\in \mathcal{C}_\eps(\Omega)\Big\}.
$$
If $\Omega'=\Omega$, we drop the dependence on the set in the notation.
Our model deviates from the one used in \cite{LL2}, because  of condition (H2) for the circulation of the admissible  fields $A$, which makes the model accounting for the anisotropy due to the underlying crystalline structure.

\begin{remark}\label{r.S}
In what follows, for every $A\in L^1(\Omega;\R^{2\times 2})$ we will denote with $S_A$ any relatively closed set of $\Omega$ such that
\[
\frac{1}{\eps}\F_\eps(A,S_A)\leq F_\eps(A) +\eps.
\]
Clearly, $S_A$ is not uniquely determined.
\end{remark}

\begin{remark}\label{r-nucleo}
	The energy $\F_\eps$ does not depend on the values of the field $A$ in the $\lambda\eps$-neighbor{-}hood of the set $S$. Indeed, for every $u\in W_0^{1,1}(B_{\lambda \eps}(S))$, $(A+ \nabla u, S)\in \mathcal{C}_\eps(\Omega)$ and its energy remains unchanged
	\[
	\F_\eps(A+\nabla u, S) = \F_\eps(A, S).
	\]
\end{remark}

\begin{remark}\label{rem-invariance-rotation}
	The energy $F_\eps$ is rotationally invariant, while condition (H2) (in view of Remark \ref{rem-invariance-rotation0}) for the circulation of the admissible competitor is invariant under right multiplication by rotations.
\end{remark}

\subsection{Main results}
We aim at proving a $\Gamma$-convergence result for the functionals $F_\eps$, thus deriving an interfacial energy for grain boundaries which accounts for the presence of crystal defects that may resolve incompatibilities. In this limiting model the polycrystalline configurations will be represented by fields $A$ that takes values in $SO(2)$, called \emph{micro-rotations} (see the definition below), which describe the locally constant orientation of the underlying lattice.

\begin{definition}\label{def-micro-rotation}
	We say that  a matrix field $A\in SBV(\Omega,SO(2))$ is a \emph{micro-rotation}
	if
	$$
	DA^i= \left(A^{i,+}- A^{i,-}\right)\otimes \nu_A \Ha^{1}\LL J_A,\qquad i=1,2,
	$$
	where $J_A$ is the jump set of $A$, $A^1$ and $A^2$ the rows of $A$, and $\nu_A$ the normal vector to $J_A$ (pointing towards the value $A^+$).
\end{definition}

Observe that a micro-rotation is not necessarily a piecewise constant field on a \emph{finite} Caccioppoli partition.
Nevertheless it can be proved (see, e.g., \cite{Bellettini-Chambolle-Goldman}) that a micro-rotation can be approximated in energy by  polyhedral partitions (i.e., fields which are piecewise constant on finite Caccioppoli partitions of subdomains with polyehdral boundaries).

\medskip

The main results of the paper are the following.

\begin{theorem}\label{t.main}
	For any sequence $(A_\eps)_{\eps>0}\subset L^1(\Omega;\R^{2\times2})$ such that
	\[
	\limsup_{\eps\to 0^+} F_\eps(A_\eps)<+\infty,
	\]
	there exist a subsequence (not relabeled) and
	a micro-rotation $A$ such that
	\begin{equation}\label{eq-conv}
		A_\eps \chi_{\Omega\setminus B_{\lambda\eps}(S_{A_\eps})}\to A \qquad \textup{in }\quad L^1(\Omega).
	\end{equation}
	Moreover, the functionals $F_\eps$ $\Gamma$-converge with respect to the  topology induced by the convergence in \eqref{eq-conv} to
	\begin{equation}\label{eq-limit0}
		F_0(A)=\begin{cases}\displaystyle{
				\int_{J_A}\Phi(A^+,A^-,\nu_A) \,d\Ha^{1}},\qquad & \hbox{ if $A$ is a microtation},\\
			+\infty & otherwise,
		\end{cases}
	\end{equation}
	for a suitable continuous energy density $\Phi:SO(2)\times SO(2)\times \mathbb S^1\to [0,+\infty)$  satisfying for all $R^-,R^+\in SO(2)$, and  $n\in \mathbb{S}^1$ the bounds
	\begin{equation}\label{eq-log-bounds0}
		C_1|R^--R^+| |\log|R^--R^+||\leq \Phi(R^-,R^+,n)\leq C_2|R^--R^+|(|\log|R^--R^+||+1),
	\end{equation}
	with $C_1,C_2>0$ dimensional constants.
\end{theorem}

\begin{remark}
	The convergence in \eqref{eq-conv} cannot be upgraded to a $L^1$ convergence for the matrix field $A_\eps$. Indeed, as noticed in Remark \ref{r-nucleo}, the energy $F_\eps$ remains unchanged if we pass to $A'_\eps:=A_\eps + \nabla u_\eps$, with $u_\eps \in C^\infty_c(B_{\lambda\eps}(S))$, arbitrary (note that, since $\Curl A'_\eps = \Curl A_\eps$, the circulations of $A'_\eps$ remain the same).
\end{remark}

\begin{remark}
	We remark that the above result is also true under more general assumptions for the energy $\F_\eps$. In particular we can replace the elastic energy $\El^\eps$ defined above with
		$$
	\El^\eps(A,S;\Omega'):= \int_{\Omega'\setminus B_{\lambda\eps}(S)} W(A(x))\, dx,
	$$
where the energy density $W$ satisfies the following properties:
\begin{itemize}
	\item[i)]  $W(AR)=W(A)$ for all $A\in \R^{2\times 2}$ and for all $R\in SO(2)$;
	\item[ii)] There are constants $c_1,c_2>0$ such that
	$$
	c_1 \dist(A,SO(2))\leq W(A)\leq c_2 \dist(A,SO(2))\qquad \forall A\in \R^{2\times 2}.
	$$
\end{itemize}
	Clearly with this choice the limiting grain boundary energy  density $\Phi$ will depend on $W$.
	Note that condition i) expresses the invariance of the model under change of variable for the material body represented by the domain $\Omega$ (see also Remark \ref{rem-invariance-rotation}).
\end{remark}

\subsection{Scheme of the proof and structure of the paper}
The proof of Theorem \ref{t.main} exploits the many deep ideas introduced by Lauteri and Luckhaus in \cite{LL2}. The key point of our proof is the asymptotic analysis of a cell-problem formula, which allows to characterize the $\Gamma$-limit of the generalised Lauteri-Luckhaus energies.

The proof of Theorem \ref{t.main} is divided in different sections and is made of several steps.
In Section 3 and 4 we prove the compactness part of the statement. In particular, the first section is devoted to a detailed revision of the refinement strategy proposed  in \cite{LL2}, making sure that the resulting field $A$ remains close in $L^1$ to the original one, with local estimates in energy.
In Section 4 we prove the compactness, first in $BV$, by using the rigidity of incompatible fields by M\"uller-Scardia-Zeppieri \cite{MSZ} and then in $SBV$ deducing the structure of micro-rotation by the optimal lower bound  estimate by Lauteri-Luckhaus \cite{LL2}.

Section 5 is devoted to the cell-problem formula.
First by means of a classical monotonicity argument we show that existence of the thermodynamic limit of the optimal interfacial energy  with fixed boundary conditions.
Then we show, with a construction inspired to the one proposed by Lauteri and Luckhaus in the case of symmetric small angle grain boundary, an upper bound for this thermodynamic limit for any pair of rotations $R^-$ and $R^+$.

Section 6 is devoted to show that the thermodynamic limit agrees with the $\Gamma$-liminf of a straight interface. This requires to show that optimal sequences can be modified in order to achieve the boundary conditions in the cell problem formula.
Here the main difficulty is due to the curl-constraint. Indeed, the use of cut-off functions is not allowed. We overcome this obstruction showing that incompatibilities concentrate along the interface and interpolating the boundary conditions with an explicit construction that uses the rigidity estimates for incompatible fields.

Finally, in Section 7 we show the $\Gamma$-convergence result. The cell-problem formula is the basis for the $\Gamma$-liminf inequality, via a blowup argument, and for the $\Gamma$-limsup inequality, via a density argument.

In two appendices, we recall some simple covering arguments and the Whitney-extension theorem needed in the first part of the paper.

\section{Asymptotically equivalent matrix fields}\label{s.refinement}
We start with the modification procedure of compatible matrix fields $A$ made by Luckhaus-Lauteri \cite{LL2}. We do it in the following three lemmas, producing first bounded fields with controlled Curl, and then harmonic fields.
In what follows $C>0$ will denote constants that may change from line to line and can depend on the parameters $\tau, \lambda$ of the energy. Moreover, $\Omega\subset\R^2$ is a bounded connected open set and for every $\sigma>0$ we set $$\Omega_{\sigma}:= \big\{x\in \Omega : \dist(x,\partial\Omega)>\sigma\big\}.$$

\begin{lemma}\label{l.Linfty}
	There exists a constant $C>0$ such that for every open set $\Omega_0\subset\subset\Omega$, if $0<\eps<C^{-1}\dist(\Omega_0, \partial \Omega)$ and $(A,S)\in \mathcal{C}_\eps(\Omega)$, then there exists $(A', S')\in \mathcal{C}_\eps(\Omega_0)$ satisfying
	\begin{itemize}
		\item[i)] $\Curl A'=\Curl A$ in $\Omega_0$, 
		\item[ii)] $\|A'\|_{L^\infty(\Omega_0\setminus S')}\leq C$,
		\item[iii)] $\|A'-A\|_{L^2(\Omega_0\setminus B_{\lambda\eps}(S'))}^2+ \|A'-A\|_{L^1(\Omega_0)}+ \F_\eps(A', S';\Omega_0) \leq C \F_\eps(A,S)$.
	\end{itemize}
\end{lemma}

\begin{proof}
	We first show that for every $\eta>0$ there exists a pair $(A_\eta,S_\eta)\in \mathcal{C}_\eps(\Omega_1)$ with $A_\eta\in C^\infty(\Omega_1;\R^{2\times 2})$, for some $\Omega_0\subset\subset\Omega_1 \subset\subset\Omega$, such that
		\begin{equation}\label{prop-regularization}
			\|A-A_\eta\|_{L^1(\Omega_1)}+\|A-A_\eta\|_{L^2(\Omega_1\setminus S_\eta)}^2+\left|\mathcal F_\eps(A,S;\Omega_1)-\mathcal F_\eps(A_\eta,S_\eta;\Omega_1) \right|\leq \eta.
		\end{equation}
Indeed we can regularise $A$ by convolution and set
		\[
A_\eta(x):= (A\chi_{\Omega_1})\star \rho_\delta(x),\qquad S_\eta:= \overline{B_{\delta}(S)}\cap\Omega_1,
\]
where $\rho_\delta$ is a standard radial symmetric convolution kernel with support in $B_\delta$.
By the continuity of the convolution it is easy to see that for  $\delta$ small enough \eqref{prop-regularization} is satisfied and by the linearity of the convolution $(A_\eta,S_\eta)\in \mathcal{C}_\eps(\Omega_1)$.
With little abuse of notation we therefore assume that $A$ is continuous.

	We consider the sets
	\begin{gather*}
		E:=\left\{x \in \Omega_{3\eps} : \;\sup_{r\in (0,3\eps)}
		\frac{1}{|B_r|}\int_{B_r(x)\setminus B_{\lambda \eps}(S) }\dist^2(A(y), SO(2)) \, dy> 2\right\},\\
		E' := \left\{ \bigcup_{x\in \Omega_{3\eps}} B_\rho (x) : B_\rho(x)\subset E, \; \rho \geq \eps \right\}\subset E,\\
		\tilde\Omega:= \Big\{x \in \Omega_0\cap E: \dist(x, E') >6\eps, \; \dist(x, S)> (\lambda +6) \eps\Big\}\subset E.
	\end{gather*}
	These three sets are open, because $E$ is the strict super-level set of a lower semi-continuous function.
	Moreover, by the maximal function estimate (see, e.g., \cite{Stein}) we have that
	\begin{align}\label{eq-misura-E}
		|E| &\leq
		C\int_{\Omega_{3\eps}\setminus B_{\lambda \eps}(S)}\dist^2(A(x), SO(2))\;dx \leq C \El^\eps(A,S).
	\end{align}
We set $W:= E\setminus\overline E'$ and redefine  $A$ in $U:=W\cap \tilde\Omega$ by a Whitney-type extension lemma, as a consequence of the following three facts:
\begin{itemize}
\item[(W1)] $B_\eps(x) \setminus W\supset B_{\eps}(x) \setminus E \neq \emptyset$ for every $x\in U$, because otherwise $B_{\eps}(x)\subset E$ would imply $x\in E'$, against $U\cap E'=\emptyset$,
\item[(W2)] $\Curl A =0$ in the sense of distributions in $B_{6\eps}(U)$, because $\dist(x,S)>6\eps$ for every $x\in U$,
\item[(W3)] for every $x\in B_{6\eps}(U)\setminus W$ we have that $x\notin E'$ and $x\notin E\setminus\overline E'$, so that $x\not \in E$ and $B_{3\eps}(x)\cap B_{\lambda\eps}(S)=\emptyset$, thus for every $r\in (0,3\eps)$ we have
\begin{align*}
\mean_{B_r(x)} |A(y)| \,dy &\leq
\frac{1}{|B_r|}\int_{B_r(x)} \dist(A(y),SO(2)) dy + \sqrt{2}\\
&\leq \left( \frac{1}{|B_r|}\int_{B_r(x)\setminus B_{\lambda\eps}(S) }\dist^2(A(y),SO(2)) dy \right)^{\frac12} +\sqrt{2}\leq 2\sqrt{2}.
\end{align*}
\end{itemize}
	From (W2) for every $x\in U$ we find a potential $u$ in $B_{6\eps}(x)$ such that $\nabla u = A$; from (W1) we know that the complementary set of $W$ intersects $B_{\eps}(x)$ and $u\vert_{B_{6\eps}(x) \setminus W}$ is Lipschitz continuous by (W3). Then, the Whitney-extension of $u\vert_{B_{6\eps}(x)\setminus W}$ to the whole $B_{\eps}(x)$ (suitably patched via a partion of unity) will provide a new field
	$$
	A':\Omega_0\to \R^{2\times 2}, \qquad \big\{A'\neq A\big\} \subset B_\eps(U)\subset E,
	$$ with the following properties (the details of the Whitney-type extension argument are given in Lemma \ref{lemma-extension}):
	\begin{gather}
		\Curl (A'-A) = 0 \qquad \textup{in } \Omega_0\qquad \textup{and}\qquad
		|A'(x)|\leq C \qquad \forall\; x\in B_{\eps}(U)\label{e.limitatezza prima modifica}.
	\end{gather}
	Conclusion i) is then verified.
	Moreover, note that
	\begin{align}\label{e.distanza L1 prima}
		\|A'-A\|_{L^2(\Omega_0)}^2 &=
		\int_{B_\eps(U)} |A'(x) - A(x)|^2 dx\notag\\
		&\leq C\int_{B_\eps(U)} |A(x)|^2\, dx + C\|A'\|_{L^\infty(B_\eps(U))} |B_\eps(U)|\notag\\
		& \leq C\int_{\Omega\setminus B_{\lambda\eps}(S)} \dist^2(A(x),SO(2))\, dx + C(\|A'\|_{L^\infty(B_\eps(U))}+1) \, |U|\notag\\
		& \stackrel{\eqref{eq-misura-E}}{\leq} C\El^\eps(A,S).
	\end{align}
	We then set
	\begin{equation}\label{eq-S-tilde}
		S':= \overline{B_{(\lambda +6)\eps}(S)\cup B_\eps(E')}\cap \Omega_0.
	\end{equation}
	Note that
	$(A',S')\in \mathcal{C}_\eps(\Omega_0)$, because from $\Curl A' = \Curl A$ we deduce that
	\[
	\int_\gamma A' = \int_\gamma A \in \tau\eps\;\Z^2 \qquad \forall \;\gamma\subset \Omega_0\setminus S'\subset \Omega_0\setminus S,
	\]
	with $\gamma$ closed simple rectifiable curve.
	Moreover, ii) holds true as well, because
	\begin{align*}
		x\in \Omega_0\setminus S' \Longrightarrow
		\begin{cases}
			\textup{either }\; x\in \Omega_0\setminus (E\cup B_{(\lambda +3)\eps}(S)) &\Longrightarrow \quad |A'(x)|=|A(x)| \leq 2\sqrt{2},\\
			\textup{or }\; x\in (E\setminus S')\cap\Omega_0\subset U &\Longrightarrow \quad |A'(x)|= | A(x)| \leq C,
		\end{cases}
	\end{align*}
	where in the first instance we have used that
	if $x\in \Omega_{3\eps}\setminus (E\cap B_{(\lambda +3)\eps}(S))$, then by definition of $E$
	\begin{equation*}
		\dist^2(A(x), SO(2))\leq 2
		\quad \Longrightarrow\quad
		|A(x)|\leq \dist(A(x), SO(2)) + |I|
		\leq
		2\sqrt{2}.
	\end{equation*}

	Finally, in order to estimate the energy, we start noticing that
	\begin{equation}\label{eq-core-Sprime}
		\begin{split}
			\Ec^\eps(S') &
			\leq C|B_{(\lambda+3) \eps}(S)|
			+C|B_{\eps}(E')| \leq C \Ec^\eps(S) + C|B_{\eps}(E')|\\
			&\leq C \Ec^\eps(S) + C\El^\eps(A,S),
		\end{split}
	\end{equation}
	where we used Lemma \ref{l.covering1}: namely, for every $\kappa>0$ there exists a constant $C(\kappa)>0$ such that
	\[
	|B_{\kappa\eps}(S)|\leq C(\kappa)|B_{\eps}(S)|;
	\]
	in addition we used the following argument: by Vitali covering theorem, there exists a finite family of disjoint balls $B_{\rho_i}(x_i)$, with $B_{\rho_i}(x_i)\subset E$ and $\rho_i\geq \eps$, such that
	\[
	E' \subset\bigcup_{i} B_{5\rho_i}(x_i),
	\]
	which by \eqref{eq-misura-E} leads to
	\begin{align*}
		|B_{\eps}(E')| &\leq \sum_{i} |B_{\eps +5\rho_i}(x_i)| \leq \sum_{i} |B_{6\rho_i}(x_i)|\leq 6^2\sum_{i} |B_{\rho_i}(x_i)|\leq 6^2 |E| \leq C\El^\eps(A,S).
	\end{align*}
	For the elastic energy, since $S\subset S'$, in view of \eqref{e.distanza L1 prima} we have that
	\begin{align*}
		\El^\eps(A',  S';\Omega_0)&\leq C\int_{\Omega\setminus B_{\lambda\eps (S')}}\dist(A(x), SO(2))^2\, dx + C\|A'-A\|_{L^2(\Omega_0)}^2\\
		&\stackrel{\eqref{e.distanza L1 prima}}{\leq} C\El^\eps(A, S).
	\end{align*}
	Putting these estimates together we conclude iii).
\end{proof}

\begin{lemma}\label{l.regolarizzazione rotore}
	There exists a dimensional constant $C>0$ such that for every open set $\Omega_1\subset\subset\Omega$, if $0<\eps<C^{-1}\dist(\Omega_1, \partial \Omega)$ and $(A,S)\in \mathcal{C}_\eps(\Omega)$, then there exists $(A'', S'')\in \mathcal{C}_\eps(\Omega_1)$ satisfying
	\begin{itemize}
		\item[i)] $\|A''\|_{L^\infty(\Omega_1)}\leq C$,
		\item[ii)] $\|A''-A\chi_{\Omega\setminus B_{\lambda\eps}(S)}\|_{L^2(\Omega_1)}^2 + \F_\eps(A'', S'';\Omega_1) \leq C \F_\eps(A,S)$,
		\item[iii)] $(\Curl A'')\vert_{\Omega_1}\in L^\infty(\Omega_1)$ with
		\[
		|\Curl A''(x)|\leq \frac{C}{\lambda\eps}\chi_{S''}(x)\qquad\forall\;x\in \Omega_1.
		\]
	\end{itemize}
\end{lemma}

\begin{proof}
	Let $\Omega_0$ be any open set such that $\Omega_1\subset\subset \Omega_0\subset\subset \Omega$ and let $(A',S')\in \mathcal{C}_\eps(\Omega_0)$ be the pair of Lemma \ref{l.Linfty}.
	We set $\hat A:=A' \chi_{\Omega_0\setminus B_{\lambda \eps}(S')} + I \chi_{B_{\lambda \eps}(S')\cap \Omega_0}$ and
	\begin{equation}\label{f-lemma2-1}
		A'':=(1-\zeta) \hat A+ \zeta \hat A \star \varphi_{\lambda\eps} \quad \textup{in }\Omega_1,
		\qquad  S'':=\overline{B_{2\lambda \eps}(S')}\cap \Omega_1,
	\end{equation}
	where $\varphi_{\lambda\eps}$ is a mollifier with compact support on $B_{\lambda\eps}$ and $\zeta$ is a cut-off function with
	\[
	\zeta=1 \quad \textup{in } B_{\lambda\eps}(S'\cap \Omega_1),
	\quad \zeta=0\quad\textup{in}
	\quad \Omega_0\setminus B_{2\lambda\eps}(S'\cap \Omega_1)\quad\textup{and} \quad|\nabla \zeta|\leq \frac{C}{\lambda\eps},
	\]
	the last condition being admissible as soon as $\eps<\dist(\Omega_1, \partial \Omega_0)$.
	We start noticing that i) holds, because $\hat A$ and, hence, $A''$ are bounded as a consequence of Lemma \ref{l.Linfty} ii).

	Next observe that
	\begin{align*}
		\|A'' - \hat A\|_{L^2(\Omega_1)}^2 & = \|\zeta (\hat A\star \varphi_{\lambda\eps} - \hat A)\|_{L^2(\Omega_1)}^2\leq C\|\hat A\|^2_{L^\infty}\,|B_{2\lambda \eps}(S')\cap \Omega_0|\leq C\mathcal{F}_\eps (A',S';\Omega_0).
	\end{align*}
	Therefore, we can conclude that
	\begin{align}\label{e.stima L2}
		\|A''-A'\chi_{\Omega_1\setminus B_{\lambda\eps}(S')}\|_{L^2(\Omega_1)} &\leq
		\|A''-\hat A\|_{L^2(\Omega_1)} +\|\hat A-A'\|_{L^2(\Omega_1\setminus B_{\lambda\eps}(S'))}+\|\hat A\|_{L^2(\Omega_1\cap B_{\lambda\eps}(S'))}\notag\\
		&\leq C\mathcal{F}_\eps^{\frac12} (A',S';\Omega_0) +
		\sqrt{2}\,|B_{\lambda\eps}(S')\cap \Omega_1|^{\frac12}\leq C\mathcal{F}_\eps^{\frac12} (A',S';\Omega_0),
	\end{align}
	and
	\begin{align*}
		\|A''-A\chi_{\Omega_1\setminus B_{\lambda\eps}(S)}\|_{L^2(\Omega_1)} &\leq
		\|A''-A'\chi_{\Omega_1\setminus B_{\lambda\eps}(S')}\|_{L^2(\Omega_1)}+ \|A'-A\|_{L^2(\Omega_1\setminus B_{\lambda\eps}(S'))}\\
		&\qquad +\|A\|_{L^2(\Omega_1\cap (B_{\lambda\eps}(S')\setminus B_{\lambda\eps}(S)))}\\
		&= C\mathcal{F}_\eps^{\frac12}(A',S';\Omega_0) +\|A'-A\|_{L^2(\Omega_1)}+\\
		&\qquad +\left(C\int_{B_{\lambda\eps}(S')\setminus B_{\lambda\eps}(S)} \big(\dist^2(A, SO(2))+1\big) \,dx \right)^{\frac12}\\
		&\leq C\mathcal{F}_\eps^{\frac12}(A',S';\Omega_0).
	\end{align*}
	In particular, we infer ii), since
	\begin{align*}
		\El^\eps(A'',  S'';\Omega_1)&\leq C\int_{\Omega_1\setminus B_{\lambda\eps (S')}}\dist(A'(x), SO(2))^2\, dx + C\|A''-A'\chi_{\Omega_1\setminus B_{\lambda\eps}(S')}\|_{L^2(\Omega)}^2\\
		&\leq C\F_\eps(A', S';\Omega_1),
	\end{align*}
	and from Lemma \ref{l.covering1} also
	\begin{equation*}
		\Ec^\eps(S'') \leq C \Ec^\eps(S).
	\end{equation*}
	Finally, we can compute the Curl of $A''$ as follows:
	\begin{align*}
		\Curl A''&= \big(\hat A\star \varphi_{\lambda\eps}-\hat A\big) \nabla\zeta^\perp + \zeta  \hat A\star \nabla^\perp \varphi_{\lambda\eps}+(1-\zeta)\Curl  \hat A\notag
		\\
		&=(A' \star \varphi_{\lambda\eps}-A') \nabla\zeta^\perp + \zeta  \hat A\star \nabla^\perp \varphi_{\lambda\eps}+(1-\zeta)\Curl A'\\
		&=(A' \star \varphi_{\lambda\eps}-A') \nabla\zeta^\perp +\zeta  \hat A\star \nabla^\perp \varphi_{\lambda\eps},
	\end{align*}
	where we used that $A'$ is Curl-free outside $S'$.
	We infer in particular that
	\[
	\supp\big(\Curl A''\big)\subset \supp(\zeta)\cap \Omega_1 \subset \overline{B_{2\lambda \eps}(S')}\cap \Omega_1=S''.
	\]
	Moreover, since $\| \hat A\|_{L^\infty(\Omega_0)}\leq C$, we conclude that
	\[
	|\Curl A''(x)|\leq \frac{C}{\lambda\eps}\| \hat A\|_{L^\infty(\Omega_0)}\,\chi_{S''}(x)\leq \frac{C}{\lambda\eps}\chi_{S''}(x)\qquad \forall\;x\in\Omega_1.
	\]
\end{proof}

\begin{lemma}\label{l.armonica}
	There exists a dimensional constant $C>0$ such that for every open set $\Omega_1\subset\subset\Omega$, if $0<\eps<C^{-1}\dist(\Omega_1, \partial \Omega)$ and $(A,S)\in \mathcal{C}_\eps(\Omega)$, then there exists $(A_\harm, S_\harm)\in \mathcal{C}_\eps(\Omega_1)$ satisfying
	\begin{itemize}
		\item[i)] $\|A_\harm-A\chi_{\Omega\setminus B_{\lambda\eps}(S)}\|_{L^2(\Omega_1)}^2+
		\F_\eps(A_\harm,S_\harm;\Omega_1)\leq C \F_\eps(A,S)$;
		\item[ii)] $A_\harm$ is harmonic in $\Omega_1\setminus S_\harm$;
		\item[iii)] $(\Curl A_\harm)\vert_{\Omega_1}\in L^\infty(\Omega_1)$ with
		\[
		|\Curl A_\harm(x)|\leq \frac{C}{\lambda\eps}\chi_{S_\harm}(x)\qquad\forall\;x\in \Omega_1.
		\]
	\end{itemize}
\end{lemma}

\begin{proof}
	Let $(A'', S'')\in \mathcal{C}_\eps(\Omega_1)$ be the pair of Lemma \ref{l.regolarizzazione rotore}.
	Recall that, since $\Curl A''=0$ in $\Omega_1\setminus S''$, then $\Div({\rm Cof}A'')=0$ in $\Omega_1\setminus S''$.
	We consider the solution $z:\Omega_1\setminus S''\to\R^2$ of the following problem:
	\begin{equation}\label{f-lemma-harm1}
		\begin{cases}
			\Delta z= -\Div({\rm Cof}A''-A'')= \Div A''\quad & \hbox{in } \Omega_1\setminus S'',\\
			z\in H_0^1(\Omega_1\setminus S'',\R^2),
		\end{cases}
	\end{equation}
	and we set $$A_\harm:= A'' -\nabla z,\qquad S_\harm=S''.$$
	Clearly, by the definition of $z$ we have that $\Div A_\harm=0$ in $\Omega_1\setminus S_\harm$ and $\Curl A_\harm=\Curl A'' =0$ in $\Omega_1\setminus S''$,
	thus implying ii) and iii) (in view of Lemma \ref{l.regolarizzazione rotore}).
	Moreover, we can estimate the $L^2$ distance between $A''$ and $A_\harm$:
	\begin{equation}\label{f-lemma-harm2}
		\begin{split}
			\int_{\Omega_1\setminus S_\harm} |A_\harm- A''|^2 dx &= \int_{\Omega_1\setminus S_\harm} |\nabla z|^2 dx \leq  \int_{\Omega_1\setminus S''} |{\rm Cof}A''-A''|^2 dx\\&\leq 4\int_{\Omega_1\setminus S''}\dist^2(A'', SO(2))\,dx,
		\end{split}
	\end{equation}
	where we used that
	$$
	|{\rm Cof}M-M|^2 \leq 4\dist^2(M, SO(2))\qquad \forall\;M\in R^{2\times 2}.
	$$
	Indeed, we can write $M=R+B$ with $R\in SO(2)$ and $|B|=\dist(M,SO(2))$ and, since ${\rm Cof}M= R + {\rm Cof}B$, we get
	$$
	|M-{\rm Cof}M|^2= |B-{\rm Cof}B|^2 \leq 4 |B|^2.
	$$
	Finally, for i) we notice that $|B_{\lambda\eps}(S_\harm)| = |B_{\lambda\eps}(S'')|$ and
	\begin{align*}
		\|\dist(A_\harm, SO(2))\|^2_{L^2(\Omega\setminus B_{\lambda\eps}(S_\harm))} & \leq 2\|\dist(A'', SO(2))\|^2_{L^2(\Omega\setminus B_{\lambda\eps}(S''))} +\\
		&\qquad +2\|A_\harm-A''\|^2_{L^2(\Omega\setminus B_{\lambda\eps}(S''))}\\
		&\leq C\mathcal{F}_\eps(A,S).\qedhere
	\end{align*}
\end{proof}

\section{Compactness}
In this section we prove the compactness part of the $\Gamma$-convergence result in Theorem \ref{t.main}. To this aim we recall the notation $S_A$ introduced in Remark \ref{r.S} for any matrix field $A\in L^1(\Omega;\R^{2\times 2})$.

\begin{theorem}\label{t.compactness}
For every sequence
$(A_\eps)_{\eps>0}$ such that \[\limsup_{\eps\to0^+} F_\eps(A_\eps) < +\infty,\] there
exists a subsequence $\eps_n\to 0^+$ such that $A_{{\eps_n}}\chi_{\Omega\setminus B_{\lambda\eps_n}(S_{A_{\eps_n}})}$ converge to $A$ in $L^2(\Omega)$. Moreover, $A$ is a micro-rotation, and for every $\Omega_0\subset\subset\Omega$ it holds
\begin{equation*}
\int_{J_A\cap\Omega_0}|A^+-A^-|\; \big|\log|A^+-A^-|\big|\; d\Ha^1 \leq C\liminf_{\eps\to 0^+} F_\eps(A_\eps),\end{equation*}
where $C=C(\Omega_0,\Omega)>0$.
\end{theorem}



\subsection{Rigidity and BV-compactness}
A first step towards the proof of Theorem \ref{t.compactness} is given by providing the compactness in $BV$.
To this aim we start by recalling some simple results for functions of bounded variations.
In the following we denote the closed (coordinate) cubes with
\[
Q(x_0,r) := \left\{x \in \R^n: \|x-x_0\|_\infty \leq r \right\},
\]
where the sup-norm is
\begin{gather*}
	\|x-x_0\|_\infty := \sup_{}\big\{|(x-x_0)\cdot e_i| : i\in\{1,\ldots, n\}\big\},
\end{gather*}
with $e_i$ the standard basis in $\R^n$.

\begin{lemma}\label{abstract1}
	Let $\Omega\subset\R^n$ be open, $u:\Omega\to\R^N$ in $L^1$ and let $\lambda$ be a finite positive Borel measure in $\Omega$ such that for all cubes $Q(x,r)\subset\Omega$
	there exists a value $u_{Q(x,r)}\in \R^N$ such that
	$$
	\int_{{Q(x,r)}} |u-u_{Q(x,r)}|\,dx \leq r\lambda({Q(x,r)}).
	$$
	Then $u\in BV(\Omega,\R^N)$ and $|Du|\leq C\lambda$ for a dimensional constant $C>0$.
\end{lemma}

\begin{proof}
	For every open subsets $\Omega'\subset\subset\Omega''\subset\subset\Omega$, we consider $0<r<\frac{1}{10}\dist(\Omega', \partial \Omega''))$ and a grid of cubes $Q_i=Q(x_i,r)\subset \Omega''$ with pairwise disjoint interiors such that
	\[\Omega'\subset\bigcup_i Q_i .
	\]
	We define the piecewise constant function $u_r:\Omega'\to \R^N$
	$$
	u_r(x):=\sum_i u_{Q_i}\chi_{Q_i}(x) \qquad x\in \Omega'.
	$$
	By assumption we have
	\begin{equation}\label{L1}
		\int_{\Omega'} |u-u_r|\,dx\leq r \lambda(\Omega'').
	\end{equation}
	In particular $u_r$ converges to $u$ in $L^1(\Omega')$, because $\lambda(\Omega)<+\infty$.
	Now for every pair of cubes $Q_i$ and $Q_j$ with a common face, if we consider a cube $Q$ of side $2r$ containing both $Q_i$ and $Q_j$ (notice that by the choice of $r$ necessarily we have that $Q\subset \Omega''$), then
	$$
	|u_Q-u_{Q_i}|\leq \frac{1}{r^n}\left[\int_Q |u-u_Q|\,dx +\int_{Q_i} |u-u_{Q_i}|\,dx \right]\leq \frac{1}{r^{n-1}}\left[2\lambda(Q)+\lambda({Q_i}) \right].
	$$
	Therefore, if $S_{ij}=Q_i\cap Q_j$ is the face in common between $Q_i$ and $Q_j$, then
	$$
	|Du_r|(S_{ij})\leq r^{n-1}|u_{Q_j}-u_{Q_i}|\leq 4\lambda(Q)+\lambda({Q_i})+\lambda({Q_j}).
	$$
	From this we deduce that there exists a dimensional constant $C>0$ such that
	\begin{equation}\label{BVr}
		|Du_r|(\Omega')\leq C \lambda(\Omega'').
	\end{equation}
	Considering the $L^1$-convergence of $u_r$ to $u$ in $\Omega'$, we infer that $u\in BV(\Omega',\R^N)$, and by passing to the limit in \eqref{BVr} and using the lower-semicontinuity of the total variation we infer that
	$|Du|(\Omega')\leq C \lambda(\Omega'')$.
	By the arbitrariness of $\Omega'\subset\subset\Omega''$ we conclude that $u\in BV(\Omega,\R^N)$ with $|Du|\leq C\lambda$.
\end{proof}

A corollary of the previous lemma is the following perturbed version.

\begin{corollary}\label{abstract2}
	Let $\Omega\subset \R^n$ be an open and bounded set and $K\subset\subset\R^N$. Consider a sequence of functions
	$u_j:\Omega\to\R^N$ with $u_j \weakly u$ in $L^1(\Omega)$, a sequence of
	finite positive measures $\lambda_j$ with $\lambda_j \weakstar \lambda$ and a sequense of positive numbers $\gamma_j \downarrow 0$.
	If for all cubes $Q(x,r)\subset\Omega$ there exists a value $u_{j,Q(x,r)}\in K$ such that
	$$
	\int_{Q(x,r)} |u_j-u_{j,Q(x,r)}|\,dx \leq r\lambda_j(Q(x,r)) +\gamma_j,
	$$
	then $u_j\to u$ in $L^1_{\textup{loc}}(\Omega;\R^N)$, $u\in BV(\Omega;K)$ and $|Du|\leq C\lambda$ for a dimensional constant $C>0$.
\end{corollary}

\begin{proof}
	The proof is an immediate consequence of Lemma \ref{abstract1}. Indeed,  up to passing to a subsequence, for every $Q(x,r)\subset \Omega$ with center $x$ with rational coordinates and rational side-length $r$, there exists $u_{Q(x,r)}\in K$ such that $u_{j,Q(x,r)} \to u_{Q(x,r)}$ and
	\begin{align*}
		\int_{Q(x,r)} |u-u_{Q(x,r)}|\,dx &\leq \liminf_{j\to \infty}
		\int_{Q(x,r)} |u_j-u_{j,Q(x,r)}|\,dx\\
		&\leq \liminf_{j\to \infty} \left(r\lambda_j(Q(x,r)) +\gamma_j\right)
		\leq r\lambda(Q(x,r)).
	\end{align*}
	Therefore, we are in position to apply Lemma \ref{abstract1}. Moreover, for every $\Omega'\subset\subset \Omega$ and for any rational $r<{\rm dist}(\Omega',\partial \Omega)$, the piecewise constant approximations of $u$ (done on cubes with rational coordinates) given by $u_r := \sum_i u_{Q_i}\chi_{Q_i}$
	take values in $K$, we conclude also that $u$ itself takes its values almost everywhere in $K$. Finally, considering the corresponding approximations for $u_j$,  $u_{j,r} := \sum_i u_{j, Q_i}\chi_{Q_i}$, we have that
	\[
	\|u-u_r\|_{L^1(\Omega')} + \|u_j-u_{j,r}\|_{L^1(\Omega')} \leq C r+\gamma_j,
	\] and,
	since clearly $u_{j,r} \to u_r$ in $L^1(\Omega')$, we also infer the strong convergence of $u_j$ to $u$ in $L^1(\Omega',\R^N)$.
\end{proof}

For the proof of the $BV$-compactness we need to combine the previous corollary together with the rigidity estimate for incompatible fields proved by M\"uller, Scardia and Zeppieri \cite{MSZ} (see also  \cite{Garroni-Leoni-Ponsiglione} for the linear version of this estimate, and \cite{Conti-Garroni} for the general $n$-dimensional case).

\begin{theorem}[M\"uller, Scardia and Zeppieri \cite{MSZ}]\label{rigidity2}
	Let $\Omega$ be an open, bounded, simply connected
	and Lipschitz set in $\R^2$. Then, there exists a constant $C = C(\Omega) > 0$ with the following property: for every
	$A\in L^2(\Omega;\R^{2\times 2})$ with $\Curl A$ a vector valued measure, there exists a rotation $R\in SO(2)$ such that
	\begin{equation}\label{eq-rigidity}
		\|A-R\|_{L^2(\Omega)}\leq C\left(\|\dist(A,SO(2))\|_{L^2(\Omega)}+|\Curl A|(\Omega)  \right).
	\end{equation}
\end{theorem}

\begin{remark}
	Note that \eqref{eq-rigidity} is invariant under homotethic scalings; therefore, the constant $C(\Omega)$ depends only on the geometric shape of $\Omega$, not on its size: for examples, it is the same for all cubes.
\end{remark}

The following lemma is the core of the $BV$-compactness as shown in \cite{LL2}. For the readers convenience we give the details of the proof.

\begin{lemma}[{\cite[Proposition 3]{LL2}}]\label{l.proposition3}
	Let $U\subset \R^2$ be bounded and open, and let $A_j\in L^2(U;\R^{2\times 2})$ be a sequence such that
	$$
	\lim_{j\to +\infty} \|\dist(A_j,SO(2))\|_{L^2(U;\R^{2\times 2})}=0\qquad \hbox{and}\qquad \sup_j|\Curl (A_j)|(U)<\infty.
	$$
	Then, there exists a subsequence $A_{j'}$ strongly converging in $L^2$ to a field $A\in BV(U; SO(2))$
	and
	\begin{equation}\label{e.rotore controlla gradiente}
		|DA|\leq C |\Curl A|\,.
	\end{equation}
\end{lemma}

\begin{proof}
	We want to apply Corollary \ref{abstract2} with $n=2$, $N=2\times 2$, $u_j=A_j$, $u=A$ a $L^1$-weak limit of a subsequence of $A_j$, and $K=SO(2)$.
	First notice that, since
	\begin{equation}\label{e.facile}
		|A_j| \leq \dist(A_j,SO(2)) + \sqrt{2},
	\end{equation}
	we have that
	$A_j$ is bounded in $L^{2}(U;\R^{2\times 2})$ and therefore, up to a subsequence, it converges to some $A$ weakly in $L^1$.
	From the rigidity estimate \eqref{eq-rigidity}, we obtain that for every cube $Q=Q(x,r)\subset\Omega$ there exists a rotation $R^j_Q\in SO(2)$ such that
	$$
	\|A_j-R^j_Q\|_{L^{2}(Q;\R^{2\times 2})}\leq C\left(\|\dist(A_j,SO(2))\|_{L^{2}(Q)}+|\Curl A_j|(Q)  \right).
	$$
	Then, using H\"older inequality,
	\begin{equation}\label{Ajrotation}
		\int_Q|A_j-R^j_Q|\, dx\leq r\, C\left(\|\dist(A_j,SO(2))\|_{L^{2}(Q)}+|\Curl A_j|(Q)  \right).
	\end{equation}
	This implies that we can apply Corollary \ref{abstract2}  with
	\[
	\gamma_j =C\|\dist(A_j,SO(2))\|_{L^{2}(U)},\qquad \lambda_j= C|\Curl A_j|,
	\]
	and $\lambda$ a weak$^*$ limit of a subsequence (not relabelled) of $\lambda_j$,
	thus infering that $$A\in BV(U;SO(2)).$$
	In order to prove that $A_j \to A$ in $L^2(U;\R^{2\times 2})$, we notice that by the triangular inequality \eqref{e.facile}
	we have strong convergence of the $L^2$-norms
	\begin{align*}
		2 |U|&=\int_{U} |A|^2dx \leq   \liminf_{j\to \infty} \int_{U} |A_j|^2dx\\
		& \leq \liminf_{j\to \infty} \int_{U} \left( \sqrt{2}+ \dist(A_j,SO(2))\right)^2dx = 2 |U|,
	\end{align*}
	because $\dist(A_j,SO(2))\to 0$ in $L^2(U;\R^{2\times 2})$.

	Finally, \eqref{e.rotore controlla gradiente} is a direct consequence of Lemma \ref{abstract1} applied with $u=A$, $\lambda = C|\Curl A|$: for each square $Q(x,r)$ we set $u_{Q(x,r)}=R\in SO(2)$ to be the rotation provided by the rigidity Theorem \ref{rigidity2}; then,
	\[
	\int_{Q(x,r)}|A(y)-R|\, dy\leq 2r\,\|A-R\|_{L^{2}(Q(x,r))}\leq C r \,|\Curl A|(Q(x,r)),
	\]
	and Lemma \ref{abstract1} gives $|DA|\leq C |\Curl A|$.
\end{proof}

Joining together the previous lemmas we get the following compactness result in $BV$ for sequence of fields $A_\eps$ with equi-bounded $F_\eps$ energies.

\begin{proposition}\label{p.compactness1}
For every sequence of matrix fields $(A_\eps)_{\eps>0}$ such that \[\limsup_{\eps\to0^+} F_\eps(A_\eps) < +\infty,\] there exists a subsequence $\eps_n\to 0^+$ such that $A_{\eps_n}\chi_{\Omega\setminus S_{A_\eps}}$ converge in $L^2(\Omega;\R^{2\times 2})$ to a field $A\in BV(\Omega;SO(2))$ and $|DA|\leq C|\curl A|$ for a dimensional constant $C>0$.
\end{proposition}

\begin{proof}
Applying Lemma \ref{l.regolarizzazione rotore} to $(A_\eps, S_{A_\eps})$, for every subdomain $\Omega_1\subset\subset\Omega$ there exist  pairs $(A''_\eps, S''_\eps)\in \mathcal{C}_\eps(\Omega_1)$ satisfying both conditions of Lemma \ref{l.proposition3} in $U=\Omega_1$, because
\begin{equation}
\begin{split}
\int_{\Omega_1}\dist^2(A''_\eps,SO(2))\,dx&\leq \int_{\Omega_1\setminus B_{\lambda\eps}(S''_\eps)}\dist^2(A''_\eps,SO(2))\,dx + C|B_{\lambda\eps}(S''_\eps)|\\
&\leq C\F_\eps(A_\eps,S_{A_\eps})\leq C\eps,
\end{split}
\end{equation}
and
$$
|\Curl (A''_\eps)|(\Omega_1)\leq \frac{C}{\lambda\eps} |B_{\lambda\eps}(S''_\eps)|\leq C.
$$
Therefore, there exists a subsequence $A''_{\eps_n}$ strongly converging in $L^2(\Omega_1;\R^{2\times 2})$ to a matrix field $A\in BV(\Omega_1;SO(2))$ with $|DA|\leq C|\Curl(A)|$ in $\Omega_1$.

From the arbitrariness of $\Omega_1$, by a diagonal argument there exists a subsequence (not relabeled) and a field $A\in BV_\loc(\Omega;SO(2))$ such that $A''_{\eps_n}$ converge $L^2_\loc(\Omega;\R^{2\times 2})$ to $A$.
In particular, from Lemma \ref{l.regolarizzazione rotore} we deduce that $A_{\eps_n}\chi_{\Omega\setminus B_{\lambda\eps}(S_{A_{\eps_n}})}$ converge $L^2_\loc(\Omega;\R^{2\times 2})$ to $A$.
Moreover, we conclude the strong convergence in $L^2(\Omega;\R^{2\times 2})$ because of the convergence of the norms:
\begin{equation*}
\begin{split}
\|A\|_{L^2(\Omega)} &= \left(2|\Omega|\right)^{\frac12} \leq \liminf_{n\to+\infty} \|A_{\eps_n}\chi_{\Omega\setminus B_{\lambda\eps}(S_{{\eps_n}})}\|_{L^2(\Omega)}\\&\leq \|\dist(A_{\eps_n}, SO(2))\|_{L^2(\Omega\setminus B_{\lambda\eps}(S_{{\eps_n}})}+\left(2|\Omega|\right)^{\frac12}=\left(2|\Omega|\right)^{\frac12}.
\end{split}
\end{equation*}
Clearly, $A$ must belong to $BV(\Omega;SO(2))$ and satisfies $|DA|\leq C|\Curl(A)|$ in $\Omega$.
\end{proof}

\subsection{Logarithmic estimate and SBV-compactness}

The proof of the compactness in $SBV$ uses the main result of Lauteri--Luckhaus stated below.

\begin{theorem}[{\cite[Theorem 3, Remark 2]{LL2}}]\label{t.LL}
For every $\Omega_0\subset\subset\Omega$ open, there exist constants $C>0$ and $\delta\in (0,1)$ with the following property:
Let $A\in L^1(\Omega;SO(2))$ and $\mu$ be a finite measure in $\Omega$ such that
\begin{itemize}
\item[(i)] there exist $\eps_n\to 0^+$ and $(\tilde A_n, \tilde S_n)\in \mathcal{C}_{\eps_n}(\Omega_0)$ with $\limsup_{\eps_n\downarrow0^+} \eps_{n}^{-1}\F_{\eps_n}(\tilde A_n, \tilde S_n;\Omega_0) < +\infty$ and $\tilde A_n \to A$ in $L^1(\Omega_0)$;
\item[(ii)] $\tilde A_{n}$ are harmonic  in $\Omega_0\setminus B_{\lambda\eps}(\tilde S_{n})$;
\item[(iii)] $\mu$ is the weak$^*$ limit of the measures
\[\mu_n:=  \left(\frac{1}{\eps_n}\dist^2(\tilde A_{n},SO(2)) \chi_{\Omega_0\setminus B_{\lambda\eps_n}(\tilde S_{{n}})}  + \frac{1}{\lambda\eps_n} \chi_{B_{\lambda\eps_n}(\tilde S_{{n}})}\right)\mathcal{L}^2\LL \Omega_0,\]
\item[(iv)]
\begin{equation*}\label{eq-remark2-LL}
|\Curl \tilde A_{n}|\leq \frac{C_0}{\eps_n}\chi_{\tilde S_{A_{n}}}\qquad \hbox{in } \Omega_0.
\end{equation*}
\end{itemize}
Then, for every $x\in \Omega_0$ and for every $R>0$ such that $B_{2R}(x)\subset \Omega_0$, there exists $\rho = \rho(x,R) \in [R, 2R]$ such that
\begin{equation}\label{e:pseudolinear}
|\Curl(A)(B_{\rho}(x))|\leq C_0 \,\omega\left(\frac{\mu(\overline {B_{2R}(x)})}{4R}\right)\mu(\overline{B_{2R}(x)}),
\end{equation}
where $\omega:(0,\infty)\to (0,\infty)$ is the continuous increasing function defined by
\begin{equation}
\omega(t)=\begin{cases}
|\log(\delta)|^{-1} &\ \ \hbox{if}\ \ t\geq \delta,\\
|\log(t)|^{-1} & \ \ \hbox{if}\ \ t< \delta.
\end{cases}
\end{equation}
\end{theorem}

The missing step for the proof of Theorem \ref{t.compactness} is the following result.

\begin{proposition}\label{prop-lower-bound}
Let $(A_{\eps_n})_n\in L^1(\Omega;\R^{2\times 2})$ be a sequence such that
$$
\lim_{n\to +\infty} F_{\eps_n}(A_{\eps_n})<+\infty
\quad \text{and}\quad A_{\eps_n}\chi_{S_{A_{\eps_n}}} \stackrel{L^1}{\longrightarrow} A.
$$
Then, $A$ is a micro-rotation according to Definition \ref{def-micro-rotation} and for every $\Omega_0\subset\subset\Omega$ there exists a constant $C=C(\Omega_0,\Omega)>0$ such that
\begin{equation}
\int_{J_A\cap\Omega_0}|A^+-A^-| |\log(|A^+-A^-|)| d\Ha^1 \leq C\liminf_{n\to +\infty} F_{\eps_n}(A_\eps).
\end{equation}
\end{proposition}

\begin{proof}
We divide the proof into three steps.

\emph{Step 1: estimating $\Curl(A)$.}
We fix $\Omega_0\subset\subset \Omega$. From Lemma \ref{l.armonica} there exists a sequence $A^\harm_n$ converging to $A$ in $L^1(\Omega_0)$, such that $F_{\eps_n}(A^\harm_n;\Omega_0)\leq C F_{\eps_n}(A_{\eps_n})$ and such that (i)-(iv) of Theorem \ref{t.LL} hold. We then deduce that
\begin{equation}\label{e:pseudolinear-2}
|\Curl(A)(B_{\rho}(x))|\leq C_0\,\omega\left(\frac{\mu(\overline {B_{2R}(x)})}{4R}\right)\mu(\overline{B_{2R}(x)}),
\end{equation}
where $\mu$ is any weak$^*$ limit of the energy density of a subsequence of $A^\harm_n$ (not relabeled).

\emph{Step 2: dimension of $\Curl(A)$.}
Let $\Theta^*_1$ be the upper $1$-density of  $\mu$, i.e.,
$$
\Theta^*_1(x):=\limsup_{r\to 0^+} \Theta_1(x,r),\qquad\text{where } \Theta_1(x,r):=\frac{\mu(\overline{B_r(x)})}{2r},\qquad x\in \Omega_0.
$$
The set
\[
\Sigma:=\{x\in \Omega_0:\Theta^*_1(x)>0\}
=\bigcup_{n\geq 1}\left\{x\in \Omega: \Theta^*_1(x)>\frac{1}{n}\right\}
\]
is $\Ha^1$ $\sigma$-finite, because by Besicovitch Theorem
\[
\Ha^1\left(\left\{x\in \Omega_0: \Theta^*_1(x)>\frac{1}{n}\right\}\right)\leq n\mu(\Omega)<+\infty.
\]
We prove that
$|\Curl(A)|(\Omega_0\setminus \Sigma) =0$.
To this purpose, for every $k\in \N\setminus\{0\}$ and $\delta>0$ consider the sets
\begin{multline*}
E_{k,\delta}= \left\{x\in \Omega_0:\Theta_1(x,2r)<\delta\quad\textup{and}\quad
|\Curl(A)(B_r(x))|\geq |\Curl(A)|(B_r(x))/2\quad\forall\,r\leq \frac2k \right\}.
\end{multline*}
By definition we have that $E_{k,\delta}\subset E_{k+1,\delta}$ for every $k$ and $E_{k,\delta}\subset E_{k,\delta'}$ for every $\delta <\delta'$.
Moreover,
\begin{equation}
\bigcap_{\delta>0}\bigcup_{k\geq 1} E_{k,\delta} \supset	\left\{x:\Theta^*_1(x)=0\;\textup{and} \  |\rho_A(x)|\geq 1/2 \right\}\cap \mathcal{N},
\end{equation}
where $|\Curl(A)|(\mathcal{N})=0$ and
\[
\rho_A (x)=\liminf_{r\to 0}\frac{\Curl(A)(B_r(x))}{|\Curl(A)|(B_r(x))}.
\]
In particular, since by Radon-Nikodym Theorem $|\rho_A|=1$ for $|\Curl(A)|$-a.e.~$x$, we have that
\begin{equation}\label{e.limite}
\begin{split}
|\Curl(A)|(\Omega\setminus \Sigma) &= |\Curl(A)|\big(\big\{x:\Theta^*_1(x)=0\;\textup{and} \  |\rho_A(x)|\geq 1/2 \big\}\big)\\
&\leq \lim_{\delta\to0} \lim_{k\to +\infty}|\Curl(A)|(E_{k,\delta}).
\end{split}
\end{equation}
We now apply Lemma \ref{l.covering1} to the measure $\nu= |\curl(A)|\LL \bigcup_{x\in E_{k,\delta}}B_{1/(2k)}(x)$: there exists a subfamily $\{B_{1/(2k)}(x_i)\}_{i\in I}$ such that $\{B_{2/k}(x_i)\}_{i\in I}$ are disjoint, and
\begin{align*}
|\Curl(A)|\left(\bigcup_{i\in I}B_{1/k}(x_i)\right)
&\geq\eps_0 |\Curl(A)|\left( \bigcup_{x\in E_{k,\delta}}B_{1/(2k)}(x)\right)\\
&\geq \eps_0 |\Curl(A)|\left(  E_{k,\delta}\right).
\end{align*}
Then, if $\rho(x_i,\frac{1}{k})\in [{1/k},{2/k}]$ is the radius provided in Theorem \ref{t.LL}, we infer that
\begin{align*}
\eps_0 |\Curl(A)|\left(  E_{k,\delta}\right) & \leq |\Curl(A)|\left(\bigcup_{i\in I}B_{1/k}(x_i)\right)\\
&= \sum_{i\in I} |\Curl(A)|(B_{1/k}(x_i))\\
&\leq \sum_{i\in I} |\Curl(A)|(B_{\rho(x_i,\frac{1}{k})}(x_i))\\
&\leq 2\sum_{i\in I} |\curl(A)(B_{\rho(x_i,\frac{1}{k})}(x_i))| \leq 2\sum_{i\in I}  \omega\Big(\Theta_1\left(x,\frac{2}{k}\right)\Big) \mu\Big(\overline{B_{\frac{2}{k}}(x_i)}\Big)\\
&\leq 2\omega(\delta) \mu(\R^n).
\end{align*}
From \eqref{e.limite} and the arbitrariness of $\delta$,  we then infer that $|\curl(A)|(\Omega\setminus \Sigma)=0$.

\emph{Step 3: Structure Theorem.}
Recall that  $A\in BV(\Omega;SO(2))$, $|DA|\leq C |\Curl(A)|$ by Proposition \ref{p.compactness1} and $|\Curl(A)|$ is supported on the $1$-dimensional set $\Sigma$. Thus, we deduce from the structure theorem for $BV$ functions that $DA$ has only the jump part and the jump set $J_A$ satisfies $\Ha^1(J_A\setminus \Sigma) = 0$. Therefore,
\[
\Curl A= (A^+-A^-)\nu_A^\perp \Ha^1\res J_A,
\]
and for $\Ha^1$-a.e. $x\in J_A$ we have that
\begin{align*}
|A^+(x)-A^-(x)| &= \lim_{r\to 0}\frac{|DA|(B_{r/2}(x))}{r} \\
&\leq C\limsup_{r \to 0}\frac{|\Curl(A)|(B_{\rho(x,r)}(x))}{2\rho(x,r)}\\
&=C \limsup_{r \to 0}\frac{|\Curl(A)(B_{\rho(x,r)}(x))|}{2\rho(x,r)}\\
&\leq C\limsup_{r \to 0}\omega(\Theta_1(x,2r))\frac{\mu(\overline{B_{2r}(x)})}{4 r}\\
& = C \omega(\Theta^*_1(x))\Theta^*_1(x),
\end{align*}
where we have used that both $|DA|$ and $|\curl(A)|$ are absolutely continuous with respect to $\Ha^1\res J_A$. Considering that the function
\[f(t)= \begin{cases}
-\frac{t}{\log t} & t\in (0,\delta),\\-\frac{t}{\log \delta} & t\in [\delta,+\infty),
\end{cases}\] is one-to-one, we can invert the relation above and infer
\[
f^{-1} (C^{-1}|A^+(x)-A^-(x)| ) \leq \Theta^*_1(x).
\]
Finally we remark that $f^{-1}(s) \geq C s |\log s|$ and therefore
\[
|A^+-A^-| |\log(|A^+-A^-|)| \Ha^1\LL J_A \leq C\Theta^*_1(x)\Ha^1\LL J_A  \leq C\mu.
\]
Integrating over $\Omega_0$ and using that
$$
\mu(\Omega_0)\leq \liminf_{n\to\infty} F_{\eps_n}(A^\harm_n;\Omega_0)\leq C\liminf_{\eps\to 0^+} F_\eps(A_\eps),
$$
we conclude the proof.
\end{proof}

\section{The cell problem formula}\label{sec-cell-problem}

In this section we define the energy per unit of length of a straight interface with normal vector $n\in \mathbb{S}^1$
separating two grains with constant rotations $R^-,R^+\in SO(2)$.
To this aim we need to study two sorts of cell problems, suited for the $\Gamma\hbox{-}\limsup$ and the  $\Gamma\hbox{-}\liminf$  inequalities, respectively: the former is a thermodynamic limit for minimizers with fixed boundary conditions; the latter accounts for the computation of the asymptotic optimal energy of
a single straight interface.

In view of Remark \ref{rem-invariance-rotation} it is enough to consider $n=e_1$, the general line tension will be obtained by a change of variable.

For the purpose of this section, we introduce the following notation:
 \begin{gather*}
 I_{R^-,R^+}(x):=R^-\chi_{\{x_1 <0\}}(x) +R^+\chi_{\{x_1\geq0\}}(x).
 \end{gather*}
Furthermore, we recall the following  scaling property:
consider a pair $(A,S)\in \mathcal{C}_\eps(Q_r)$, where $Q_r$ is the open square centered at $0$ of side $2r$, i.e.,
\[
Q_r=\Big(-r,r\Big)\times \Big(-r,r\Big),
\]
then
\begin{equation}\label{eq-scaling-property}
\begin{split}
&(\tilde A,\tilde S)\in \mathcal{C}_{\eps/\rho}(Q_{r/\rho}),\qquad\tilde A(y):=A(\rho y), \quad\tilde S:=\rho^{-1} S,\quad \rho>0,\\
&B_{\lambda \eps}(S)=\rho B_{\frac{\lambda\eps}{\rho}}(\tilde S) \quad\textup{and}\quad	\F_\eps(A,S, Q_r)=\rho^2 \F_{\eps/\rho}(\tilde A, \tilde S, Q_{ r/\rho}).
\end{split}
\end{equation}
In particular, if $r=1$ and $\rho=\eps$, we obtain that
\begin{equation}\label{eq-scaling-property-2}
	\eps^{-1}\F_\eps(A,S, Q_1)=\eps \F_{1}(\tilde A, \tilde S, Q_{ 1/\eps})
\end{equation}
and, as a consequence,
\begin{equation}\label{eq-scaling-property-3}
	F_\eps(A, Q_1)=\eps F_{1}(\tilde A, Q_{1/\eps}).
\end{equation}

\subsection{The thermodynamic limit}


We define the optimal transition energy at scale $L>0$ as follows
\begin{equation}\label{eq-cell-problem-eps}
	\begin{split}
\psi(R^-,R^+,L):=\inf\Big\{F_1(A,Q_L)\,:\ &A\in L^1_{loc}(\R^2;\R^{2\times 2})\,,\ \\
&A = I_{R^-,R^+} \hbox{ in } \R^2\setminus Q_{L'}, \hbox{for some } L'<L\Big\}.
	\end{split}
\end{equation}

The next proposition shows that there exists the thermodynamic limit of the optimal transition functions as the size of the square tends to infinity.

\begin{proposition}\label{prop-cell-L}
For every $R^+,R^-\in SO(2)$ 
the following limit exists
\begin{equation}\label{eq-cell-limit-2}
\psi_\infty(R^-,R^+):= \lim_{L\to +\infty}\frac{1}{2L}\psi(R^-,R^+,L).
\end{equation}
\end{proposition}

\begin{proof}
For every $L>0$ and $\eta>0$, we consider a matrix field $A_L:\R^2\to \R^{2\times 2}$ such that $A_L=I_{R^-,R^+}$ in $\R^2\setminus Q_{L'}$, for some $L'<L$, and
\begin{equation}\label{eq-cell-L}
\psi(R^-,R^+,L)\geq F_1(A_L, Q_L)-2\eta\,L.
\end{equation}
Let $\hat A_L$ denote the field obtained by the periodic extension in the vertical direction
\begin{equation}\label{eq-AL-estesa}
\hat A_L(x_1,x_2+2kL)=A_L(x_1,x_2)\qquad \forall (x_1,x_2)\in \R \times[-L,L], \quad \forall\:\, k\in\Z.
\end{equation}
For every $M>2L$, we set $M'=\lfloor M/L\rfloor L-L$ and $A_{M,L}=\hat A_L\chi_{Q_{M'}}+ I_{R^-,R^+}(1-\chi_{Q_{M'}})$. Then,
\begin{equation}
\begin{split}\label{eq-extension-estimate}
F_1(A_{M,L}, {Q_M})&=
F_1(A_{M,L}, (-L,L)\times (-M,M))\\
&\leq \lfloor M/L\rfloor F_1(A_L, Q_L) +CL\\
&\leq \lfloor M/L\rfloor (\psi(R^-,R^+,L)+L\eta ) +CL,
\end{split}
\end{equation}
where the contribution $CL$ accounts for the measure of the tubular neighborhood of the support of $\Curl A_{M,L}$ in $Q_M\setminus Q_{M'}$.
Thus we deduce that
\begin{equation}\label{e.stimozzo}
\begin{split}
\frac{1}{2M}F_1(A_{M,L}, {Q_M})\leq  &	\frac{1}{2M}(\lfloor M/L\rfloor (\psi(R^-,R^+,L)+2L\eta ) +CL)
\end{split}
\end{equation}
and hence
\begin{align}\label{eq-cell-limsup}
\limsup_{M\to +\infty}\frac1{ 2M}\psi(R^-,R^+,M) &\leq \limsup_{M\to +\infty}
\frac{1}{2M}F_1(A_{M,L},  Q_M)\notag\\
&\leq\frac1{2L} \psi(R^-,R^+,L) +\eta.
\end{align}
Taking the liminf as $L\to +\infty$, we deduce that
\begin{equation}\label{eq-confronto}
\limsup_{M\to +\infty}\frac1{2M}\psi(R^-R^+,M)\leq \liminf_{L\to +\infty}\frac1{2L}\psi(R^-,R^+,L) +\eta,
\end{equation}
and by the arbitrariness of $\eta$ we conclude the existence of the limit $\psi_\infty$.
\end{proof}




Next we provide an upper bound for $\psi_\infty(R^-,R^+)$ which is optimal in the scaling.
To this aim we revisit the construction done by Lauteri--Luckhaus in \cite{LL2}, and we extend to our anisotropic model and to any pair $(R^-,R^+)$ not necessarily symmetric with respect to $e_2$.

\subsection{Construction with the optimal Read and Shockley scaling}

We present a construction for strain fields capable of describing the phase transition between different lattice configurations along an interface.  The construction is quite flexible, and depending on the choice of the parameters can be adapted to different pairs of configurations: rotated to rotated, incompatible lattices, rotated to incompatible.

In this section we denote $Q_{a,b}:=[-a,a]\times [-b,b]$ the rectangle with sides $2a, 2b>0$. We also denote as $R_\beta$ the rotation of angle $\beta\in \R$ given by
\begin{equation}
	R_\beta=\left(\begin{matrix}
		\cos\beta & -\sin\beta\\
		\sin \beta & \cos\beta
	\end{matrix}\right).
\end{equation}
Given the parameters $\eta, \mu\in \R$ we define the shear field
\begin{equation}
	D_{\eta,\mu}=\left(\begin{matrix}
		1 & \mu\\
		0 & \eta
	\end{matrix}\right).
\end{equation}

\subsubsection{Compression to compression}

Consider the rectangles $Q:=Q_{l,h}$, $Q^\prime:= Q_{r,\rho}$ with $r,\rho<\operatorname{min}\{h,l\}$, and let $\eta,\tilde \eta, \mu,\tilde\mu\in \R$ be given parameters. We divide $Q\smallsetminus Q^\prime$ in $8$ regions $\Delta_i$, with $i=1,\dots,8$, as in Figure \ref{figuraB}.

\begin{figure}[htb]
	\fontsize{6}{4}{
		\def\svgwidth{400pt}{\footnotesize
\begingroup%
  \makeatletter%
  \providecommand\color[2][]{%
    \errmessage{(Inkscape) Color is used for the text in Inkscape, but the package 'color.sty' is not loaded}%
    \renewcommand\color[2][]{}%
  }%
  \providecommand\transparent[1]{%
    \errmessage{(Inkscape) Transparency is used (non-zero) for the text in Inkscape, but the package 'transparent.sty' is not loaded}%
    \renewcommand\transparent[1]{}%
  }%
  \providecommand\rotatebox[2]{#2}%
  \newcommand*\fsize{\dimexpr\f@size pt\relax}%
  \newcommand*\lineheight[1]{\fontsize{\fsize}{#1\fsize}\selectfont}%
  \ifx\svgwidth\undefined%
    \setlength{\unitlength}{663.30706787bp}%
    \ifx\svgscale\undefined%
      \relax%
    \else%
      \setlength{\unitlength}{\unitlength * \real{\svgscale}}%
    \fi%
  \else%
    \setlength{\unitlength}{\svgwidth}%
  \fi%
  \global\let\svgwidth\undefined%
  \global\let\svgscale\undefined%
  \makeatother%
  \begin{picture}(1,0.35470085)%
    \lineheight{1}%
    \setlength\tabcolsep{0pt}%
    \put(0,0){\includegraphics[width=\unitlength,page=1]{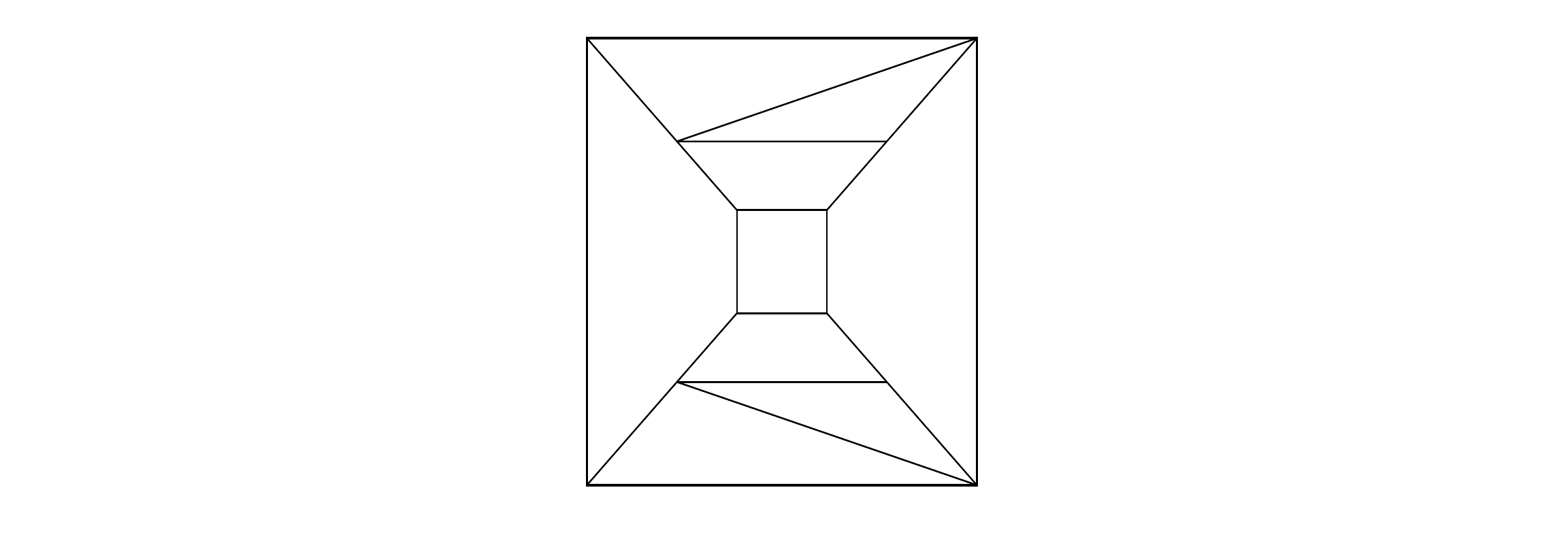}}%
    \put(0.43164383,0.28991378){\makebox(0,0)[lt]{\lineheight{0.40000001}\smash{\begin{tabular}[t]{l}$\Delta_1$\end{tabular}}}}%
    \put(0.54421136,0.17815609){\makebox(0,0)[lt]{\lineheight{0.40000001}\smash{\begin{tabular}[t]{l}$\Delta_8$\end{tabular}}}}%
    \put(0.5474507,0.08178532){\makebox(0,0)[lt]{\lineheight{0.40000001}\smash{\begin{tabular}[t]{l}$\Delta_5$\end{tabular}}}}%
    \put(0.48671286,0.12227725){\makebox(0,0)[lt]{\lineheight{0.40000001}\smash{\begin{tabular}[t]{l}$\Delta_4$\end{tabular}}}}%
    \put(0.44298159,0.06963775){\makebox(0,0)[lt]{\lineheight{0.40000001}\smash{\begin{tabular}[t]{l}$\Delta_6$\end{tabular}}}}%
    \put(0.52153588,0.27695637){\makebox(0,0)[lt]{\lineheight{0.40000001}\smash{\begin{tabular}[t]{l}$\Delta_2$\end{tabular}}}}%
    \put(0.4875227,0.23079559){\makebox(0,0)[lt]{\lineheight{0.40000001}\smash{\begin{tabular}[t]{l}$\Delta_3$\end{tabular}}}}%
    \put(0.40572901,0.18139544){\makebox(0,0)[lt]{\lineheight{0.40000001}\smash{\begin{tabular}[t]{l}$\Delta_7$\end{tabular}}}}%
  \end{picture}%
\endgroup%
}
		\caption[]{}
		\label{figuraB}
	}
\end{figure}
The aim is to define a field $D \in L^1_{loc}(\R^2;\R^{2\times 2})$ connecting the two shear fields  $D_{\eta,\mu}$ and $D_{\tilde\eta,\tilde\mu}$ via a construction mimicking the action of vertical dislocations and resolving the incompatibility with an optimal scaling. With that aim in mind we start defining $D$ in a single rectangle $Q$ and then proceed to extend it by periodicity to the whole $\R^2$. We point out that imposing the right boundary condition on $\partial Q$ is an important aspect, since it allows to glue the local construction without creating any new curl. A fundamental aspect of the construction is the fact that we are able to determine arbitrarily $\operatorname{curl} D$ as a function of the parameters $\eta,\tilde \eta, \mu,\tilde \mu, h$.

We want to define a field $\hat D$ as the gradient of an opportune deformation
\[
v:Q\setminus (Q'\cup\{x_2=0\})\to \R^2,
\]
that we now describe. First we define for any $\eta,\mu\in \R$ the piecewise affine function
\begin{equation}
	v_{\eta,\mu}(x_1,x_2)=
\begin{cases}
	\displaystyle{
D_{\eta,\mu}\left [x_1-\frac{\mu h}{\eta},x_2+\frac{h}{\eta}-h\right]}\qquad & \hbox{if } \{x_2>0\}\\
	\displaystyle{D_{\eta,\mu}\left[x_1+\frac{\mu h}{\eta},x_2+h-\frac{h}{\eta}\right] }\qquad & \hbox{if } \{x_2<0\}.
\end{cases}
\end{equation}
where, if $A\in \R^{2\times 2}$ and $x=(x_1,x_2)$, then we wrote the multiplication $Ax$ as $A[x_1,x_2]$.
Note that the function $v_{\eta,\mu}$ in the two half planes is the composition of the linear map associated to $D_{\eta,\mu}$ and a translation, and the latter is chosen so that  the points $(\pm \ell,\pm h)$ remain fixed. Moreover $v_{\eta,\mu}$ has constant jump along the line $\{x_2=0\}$ given by $2h(\mu,\eta-1)$.

We then define $v$:
\begin{itemize}
	\item[i)] In $\Delta_7$ we define $v(x_1,x_2)=v_{\eta,\mu}(x_1,x_2)$ and in $\Delta_8$ we define $v(x_1,x_2)=v_{\tilde\eta,\tilde\mu}(x_1,x_2)$;
	\item[ii)] In $\Delta_1, \Delta_2, \Delta_5$ and $\Delta_6$ we define $v$ as the linear interpolation on the vertices of the $\Delta_i$'s.
	\item[iii)] In $\Delta_3$ and $\Delta_4$ we interpolate along the horizontal direction obtaining:
	\begin{align*}
		v(x_1,x_2)=&\Big(x_1+\frac{hx_1}{2l}(\tilde\mu-\mu)-\frac{h^2x_1}{2lx_2}(\tilde\mu-\mu)+\frac{(\tilde\mu+\mu)}{2}(x_2-h)\Big)e_1+\\
		&\Big(\frac{hx_1}{2l}(\tilde\eta-\eta)-\frac{h^2x_1}{2lx_2}(\tilde\eta-\eta)+\frac{(\tilde\eta+\eta)}{2}(x_2-h)+h\Big)e_2,  \qquad x\in \Delta_3,
	\end{align*}
	\begin{align*}
		v(x_1,x_2)=&\Big(x_1-\frac{hx_1}{2l}(\tilde\mu-\mu)-\frac{h^2x_1}{2lx_2}(\tilde\mu-\mu)+\frac{(\tilde\mu+\mu)}{2}(x_2+h)\Big)e_1+\\
		&\Big(\frac{hx_1}{2l}(\eta-\tilde\eta)+\frac{h^2x_1}{2lx_2}(\eta-\tilde\eta)+(x_2+h)\frac{(\eta+\tilde\eta)}{2}-h\Big)e_2, \qquad x\in \Delta_4.
	\end{align*}
\end{itemize}
The field $\hat D$ is then defined as $\nabla v$, precisley:
\begin{equation}\label{def-B-anisotrop-constr}
	\begin{split}
		& \hat D(x)= \hat D_i(x):=	D_{\eta,\mu},\quad  x\in \Delta_i, \quad i\in \{1,6,7\}, \\
		&\hat D(x)= \hat D_8(x):=	D_{\tilde\eta,\tilde\mu},\quad  x\in \Delta_8,\\
		&  \hat D(x)=\hat D_2:=
		\left(\begin{matrix}
			1 +\frac{h}{2l}(\mu-\tilde\mu)& \frac{1}{2}(3\tilde\mu-\mu)\\
			\frac{h}{2l}(\eta-\tilde\eta) & \frac{1}{2}(3\tilde\eta-\eta)
		\end{matrix}\right),\quad  x\in\Delta_2,\\
		& \hat D(x)=\hat D_3:=
		\left(\begin{matrix}
			1+\frac{h}{2l}(\tilde\mu-\mu)-\frac{h^2}{2lx_2}(\tilde\mu-\mu) & \frac{1}{2}(\mu+\tilde\mu) +\frac{h^2x_1}{2lx_2^2}(\tilde\mu-\mu)\\
			\frac{h}{2l}(\tilde\eta-\eta)-
			\frac{h^2}{2lx_2}(\tilde\eta-\eta)&  \frac{1}{2}(\eta+\tilde\eta)+\frac{h^2x_1}{2lx_2^2}(\tilde\eta-\eta)
		\end{matrix}\right),\quad  x\in\Delta_3,\\
		& \hat D(x)=\hat D_4:=
		\left(\begin{matrix}
			1-\frac{h}{2l}(\tilde\mu-\mu)-\frac{h^2}{2lx_2}(\tilde\mu-\mu) & \frac{1}{2}(\mu+\tilde\mu) +\frac{h^2x_1}{2lx_2^2}(\tilde\mu-\mu)\\
			-\frac{h}{2l}(\tilde\eta-\eta)-
			\frac{h^2}{2lx_2}(\tilde\eta-\eta)&  \frac{1}{2}(\eta+\tilde\eta)+\frac{h^2x_1}{2lx_2^2}(\tilde\eta-\eta)
		\end{matrix}\right),\quad  x\in\Delta_4,\\
		& \hat D(x)=\hat D_5:= \left(\begin{matrix}
			1 +\frac{h}{2l}(\tilde\mu-\mu)& \frac{1}{2}(3\tilde\mu-\mu)\\
			\frac{h}{2l}(\tilde\eta-\eta) & \frac{1}{2}(3\tilde\eta-\eta)
		\end{matrix}\right),\quad  x\in\Delta_5,\\
		& \hat D(x)= I,\quad x\in Q'.\\
	\end{split}
\end{equation}
We observe that the following boundary conditions are satisfied
\begin{equation}\label{boundary-cond-D}
    \hat D_7=D_{\eta,\mu}, \qquad \hat D_1e_1=\hat D_6e_1=e_1, \qquad \hat D_8=D_{\tilde\eta,\tilde\mu}.
\end{equation}
Being the gradient of a function with constant jump on the segments $(-l,-r)\times \{0\} $ and $(r,l)\times \{0\}$, the support of $\operatorname{Curl}\hat D$ is contained in $Q^\prime$, and $\hat D$ has constant but non trivial circulation around $Q^\prime$ given by
\begin{equation}\label{curl-B}
	\int_{\partial Q'} \hat D=\int_{\partial Q} \hat D=2h(\mu-\tilde\mu,\eta-\tilde\eta).
\end{equation}
We now extend $\hat D$ to the infinite vertical stripe $(-l,l)\times \R$ by repeating its construction periodically in the vertical direction:
	\[
	\hat D(x_1, x_2+2kh) = \hat D(x_1,x_2), \qquad  (x_1,x_2)\in Q, \quad \forall\; k\in \mathbb{Z},
	\]
	this field still having compatible circulation since  $\hat D_1e_1=\hat D_6e_1$, hence no curl was created in gluing the construction vertically.
	
In order to extend the construction to the whole of $\R^2$ we can now simply set
 \begin{equation}
     D(x_1,x_2)=\begin{cases}
         D_{\eta,\mu}\,x, \qquad x_1< -l,\\
         \hat D(x_1,x_2), \qquad |x_1|\leq l,\\
          D_{\tilde\eta,\tilde\mu}\,x, \qquad x_1>l.
     \end{cases}
 \end{equation}
We observe that
\begin{equation}\label{supp-curl-D}
\operatorname{supp}\Curl (D) \subseteq \bigcup_{k\in \Z}[-r,r]\times[-\rho+2kh,\rho+2kh],
\end{equation}
since again, thanks to the boundary condition $D_7=D_{\eta,\mu}$ and $D_8=D_{\tilde\eta,\tilde\mu}$, no curl was created in extending $\hat D$ outside $(-l,l)\times \R$. Furthermore, since the circulation of $D$ around each rectangle $[-r,r]\times[-\rho+2kh,\rho+2kh]$ is given by \eqref{curl-B}, from \eqref{supp-curl-D} we infer that
\begin{equation}\label{circul-D}
    \int_{\gamma} D \in h(\mu-\tilde\mu,\eta-\tilde\eta) \Z,
\end{equation}
for every $1$-rectifiable closed loop $\gamma \subset \R^2\smallsetminus \operatorname{supp}\operatorname{curl}D$.

We now estimate the energy of $D$ in $Q$.
We start with the elastic energy in $\Delta_3$ and $\Delta_4$: for $i=3,4$ it holds that
\begin{equation}
		\begin{split}
			\dist^2(\hat  D_i,SO(2))\leq |\hat D_i-I|^2 &\leq C\frac{h^4}{l^2}[(\tilde\mu-\mu)^2+(\tilde\eta-\eta)^2](\frac{1}{x_2^2}+\frac{x_1^2}{x_2^4})\\
			&+C\frac{h^2}{l^2}[(\tilde\mu-\mu)^2+(\tilde\eta-\eta)^2]+(\eta+\tilde\eta-2),
		\end{split}
	\end{equation}
 hence
 \begin{equation}\label{bound-energia-D3}
		\begin{split}
	&\int_{\Delta_i}\dist^2(\hat D_i(x),SO(2))\leq \frac{h^4}{l^2}[(\tilde\mu-\mu)^2+(\tilde\eta-\eta)^2]\int_{\rho}^{h/2}dx_2\int_{-\frac{l}{h}x_2}^{\frac{l}{h}x_2}\Big(\frac{1}{x_2^2}+\frac{x_1^2}{x_2^4}\Big)dx_1\\
			&+C\frac{h^3}{l}[(\tilde\mu-\mu)^2+(\tilde\eta-\eta)^2]+lh(\eta+\tilde\eta-2)\\
   &\leq C\Big(\frac{h^3}{l}+hl\Big)[(\tilde\mu-\mu)^2+(\tilde\eta-\eta)^2]\Big|\log\Big(\frac{h}{2\rho}\Big)\Big|\\
   &+C\frac{h^3}{l}[(\tilde\mu-\mu)^2
   +(\tilde\eta-\eta)^2]+hl(\eta+\tilde\eta-2).
		\end{split}
	\end{equation}
For $\Delta_i$ with $i=2,5$ we have instead
\begin{equation}
     \dist^2(\hat D_i,SO(2))\leq |\hat D_i-I|^2=\frac{h^2}{4l^2}|\mu-\tilde\mu|^2+\frac{1}{4}|\mu-3\tilde\mu|^2+\frac{h^2}{4l^2}|\eta-\tilde\eta|^2+|1-\frac{1}{2}(3\tilde\eta-\eta)|^2,
\end{equation}
hence
\begin{equation}\label{bound-energia-D1}
    \int_{\Delta_i} \dist^2(\hat D_i,SO(2)) \leq  \frac{h^3}{4l}|\mu-\tilde\mu|^2+\frac{hl}{4}|\mu-3\tilde\mu|^2+\frac{h^3}{4l}|\eta-\tilde\eta|^2+hl|1-\frac{1}{2}(3\tilde\eta-\eta)|^2.
\end{equation}
Finally for $i=1,6,7$ we have
\begin{equation}\label{bound-energia-shear}
\int_{\Delta_i}\dist^2(D_{\eta,\mu},SO(2))\leq \int_{\Delta_i}|I-D_{\eta,\mu}|^2= hl\Big((1-\eta)^2+\mu^2)\Big),
\end{equation}
and a similar estimate also holds for $\Delta_8$.\\
To estimate the core energy of $\hat D$ we recall that support of $\operatorname{Curl}\hat D$ is contained in $Q^\prime$. Hence setting $S=Q^\prime$ we have
\begin{equation}\label{bound-core-D}
    |B_{\lambda\eps }(S)|=|B_{\lambda\eps }(S)|\leq C(r+\lambda\eps)(\rho+\lambda\eps).
\end{equation}

\subsubsection{Rotation to compression}\label{rotation to compression}

Consider as above the rectangles $Q:=Q_{l,h}$, $Q^\prime:= Q_{r,\rho}$ with $r,\rho<\operatorname{min}\{h,l\}$, and let $\eta, \mu,\beta\in \R$ be given parameters. We now divide $Q\smallsetminus Q^\prime$ in $4$ regions $\Delta'_i$, with $i=1,\dots,4$, as in Figure \ref{figuraA}.

\begin{figure}[htb]
	\fontsize{6}{4}{
		\def\svgwidth{400pt}{\footnotesize
\begingroup%
  \makeatletter%
  \providecommand\color[2][]{%
    \errmessage{(Inkscape) Color is used for the text in Inkscape, but the package 'color.sty' is not loaded}%
    \renewcommand\color[2][]{}%
  }%
  \providecommand\transparent[1]{%
    \errmessage{(Inkscape) Transparency is used (non-zero) for the text in Inkscape, but the package 'transparent.sty' is not loaded}%
    \renewcommand\transparent[1]{}%
  }%
  \providecommand\rotatebox[2]{#2}%
  \newcommand*\fsize{\dimexpr\f@size pt\relax}%
  \newcommand*\lineheight[1]{\fontsize{\fsize}{#1\fsize}\selectfont}%
  \ifx\svgwidth\undefined%
    \setlength{\unitlength}{663.30706787bp}%
    \ifx\svgscale\undefined%
      \relax%
    \else%
      \setlength{\unitlength}{\unitlength * \real{\svgscale}}%
    \fi%
  \else%
    \setlength{\unitlength}{\svgwidth}%
  \fi%
  \global\let\svgwidth\undefined%
  \global\let\svgscale\undefined%
  \makeatother%
  \begin{picture}(1,0.35470085)%
    \lineheight{1}%
    \setlength\tabcolsep{0pt}%
    \put(0,0){\includegraphics[width=\unitlength,page=1]{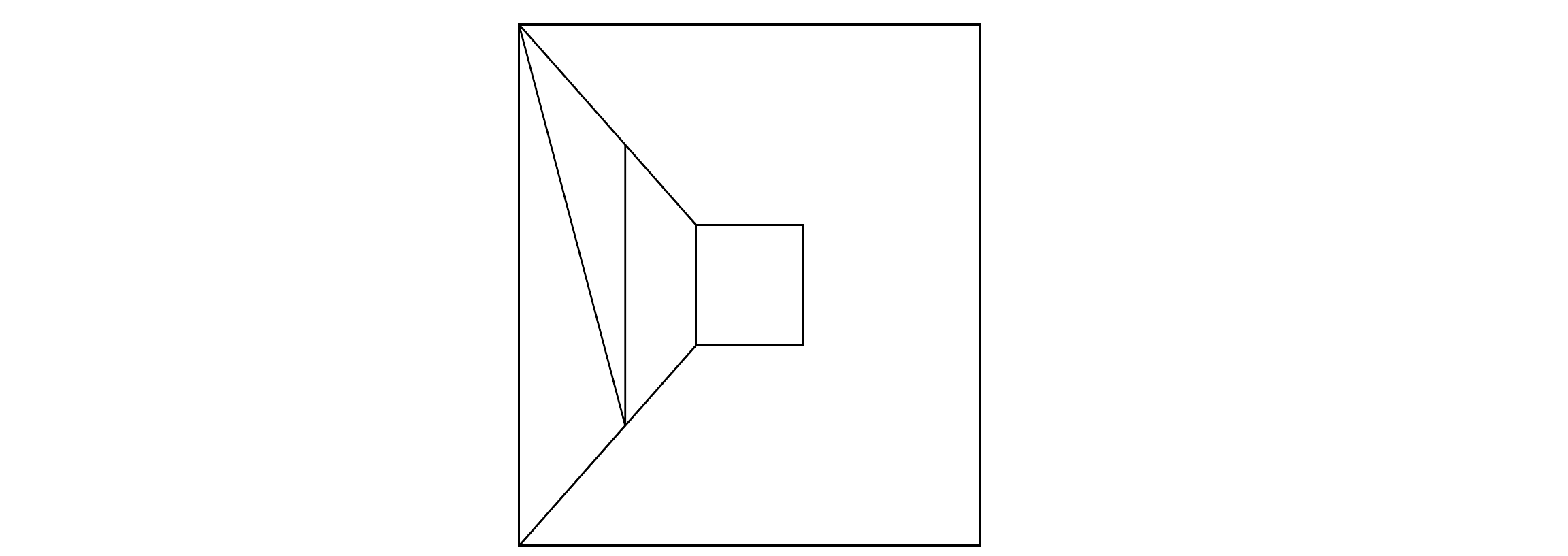}}%
    \put(0.34983765,0.12580474){\makebox(0,0)[lt]{\lineheight{0.40000001}\smash{\begin{tabular}[t]{l}$\Delta^\prime_1$\end{tabular}}}}%
    \put(0.35714285,0.25161644){\makebox(0,0)[lt]{\lineheight{0.40000001}\smash{\begin{tabular}[t]{l}$\Delta^\prime_2$\end{tabular}}}}%
    \put(0.40422081,0.17694111){\makebox(0,0)[lt]{\lineheight{0.40000001}\smash{\begin{tabular}[t]{l}$\Delta^\prime_3$\end{tabular}}}}%
    \put(0.53814936,0.1696359){\makebox(0,0)[lt]{\lineheight{0.40000001}\smash{\begin{tabular}[t]{l}$\Delta^\prime_4$\end{tabular}}}}%
  \end{picture}%
\endgroup%
}
		\caption[]{}
		\label{figuraA}
	}
\end{figure}
In the same spirit of the \textit{compression to compression} construction, the aim is now to define a matrix field $B \in L^1_{loc}(\R;\R^{2\times 2})$ connecting a rotation $R_\beta$ and a shear field $D_{\eta,\mu}$ via a construction mimicking horizontal dislocations. A fundamental aspect of the construction is the fact that we obtain $\operatorname{curl} B$ as a function of the parameters $\eta, \mu,\beta$ and $ h$.

We start defining $\hat B$ in a single $Q$. Consider the function
\[
v:Q\setminus (Q'\cup\{x_1=0,x_2>0\})\to \R^2,
\]
defined as follows:
\begin{itemize}
	\item[i)] In the region $\Delta_4$  we have
\begin{equation*}
	v(x_1,x_2)=	\begin{cases}
		D_{\eta,\mu}\,x\qquad & \hbox{if } x\in \Delta'_4\setminus \{x_1<0,x_2>0\}\\
		D_{\eta,\mu}[x_1+\xi_1,x_2+\xi_2]\qquad & \hbox{if } x\in \Delta'_4\cap \{x_1<0,x_2>0\}
	\end{cases}
\end{equation*}
where the vector $\xi=(\xi_1,\xi_2):=(-2h\frac{\mu}{\eta}\cos\beta-2h\sin\beta,\frac{2h}{\eta}\cos\beta-2h)$ is obtained by solving $D_{\eta,\mu}[-h+\xi_1,h+\xi_2]-D_{\eta,\mu}[-h,-h]=R_{\beta}2he_2$;
	\item[ii)] In $\Delta'_1$, $\Delta'_2$, and the function $v$ is the linear interpolation over the vertices of the $\Delta'_i$'s;
	\item[iii)] In $\Delta'_3$ we interpolate linearly along the vertical direction
	\begin{align*}
		v(x_1,x_2)=&\Big((\mu+\frac{l}{x_1}(\mu+\sin\beta))x_2+x_1-h(\mu+\sin \beta)\Big)e_1+\\
  &\Big((\eta +\frac{l}{x_1}(\eta-\cos\beta))x_2+h(\cos\beta-\eta)\Big)e_2,\qquad  x\in\Delta'_3.
	\end{align*}
\end{itemize}

We define the field $B=\nabla v$, which is then given by
\begin{equation}
	\begin{split}
		&\hat  B(x)=I,\quad x\in Q',\\
		&\hat B(x)=\hat B_1:=
		\left(\begin{matrix}
			1+\frac{h}{l}\mu+\frac{h}{l}\sin \beta & -\sin \beta\\
			\frac{h}{l}\eta -\frac{h}{l}\cos \beta & \cos\beta
		\end{matrix}\right),\quad  x\in \Delta'_1,\\
		&\hat  B(x)=\hat  B_2:=
		\left(\begin{matrix}
			-2\frac{h}{l}\sin\beta-2\frac{h}{l}\mu+1 & -2\sin\beta-\mu\\
			2\frac{h}{l}\cos\beta-2\frac{h}{l}\eta & 2\cos\beta-\eta
		\end{matrix}\right),\quad x\in \Delta'_2,\\
		&\hat  B(x)=\hat B_3:=
		\left(\begin{matrix}
			1-\frac{lx_2}{x_1^2}(\mu+\sin \beta) & \mu+\frac{l}{x_1}(\mu+\sin \beta)\\
			-\frac{lx_2}{x_1^2}(\eta-\cos\beta)  & \eta + \frac{l}{x_1}(\eta-\cos\beta)
		\end{matrix}\right),\quad  x\in\Delta'_3,\\
		&\hat  B(x)=\hat  B_4(x):=	D_{\eta,\mu},\quad   x\in \Delta'_4.\\
	\end{split}
\end{equation}
We observe that the following  boundary conditions are satisfied:
\begin{equation}\label{boundary-cond-B}
   \hat  B_1e_2=R_\beta e_2, \qquad \hat B_4=D_{\eta,\mu}.
\end{equation}
As before we compute the circulation of $\hat B$ around $Q^\prime$ following $\partial Q$ as a path.  With a rapid calculation we find that the clockwise circulation of $\hat B$ is given by
\begin{equation}\label{ident-circul-B}
	 \int_{\partial Q'} \hat B=\int_{\partial Q} \hat B=2h(-\mu-\sin\beta,\cos\beta-\eta).
\end{equation}

We now extend $\hat B$ to the infinite vertical stripe $(-l,l)\times \R$ by repeating its construction periodically in the vertical direction:
	\[
	\hat B(x_1, x_2+2kh) = \hat B(x_1,x_2), \qquad  (x_1,x_2)\in Q, \quad \forall\; k\in \mathbb{Z},
	\]
	this field still having compatible circulation because of the boundary conditions \eqref{boundary-cond-B}, hence no curl was created in gluing the construction vertically.\\
 To extend the construction to the whole of $\R^2$ we can now simply set
 \begin{equation}
     B(x_1,x_2)=\begin{cases}
         R_\beta \,x, \qquad x_1< -l,\\
         \hat B(x_1,x_2), \qquad |x_1|\leq l,\\
          D_{\eta,\mu}\,x, \qquad x_1>l.
     \end{cases}
 \end{equation}
We observe that
\begin{equation}\label{supp-curl-B}
\operatorname{supp}\Curl (B) \subseteq \bigcup_{k\in \Z}[-r,r]\times[-\rho+2kh,\rho+2kh],
\end{equation}
since again, thanks to the boundary condition $B_1e_2=R_\beta e_2$ and $B_4=D_{\eta,\mu}$, no curl was created in extending $\hat B$ outside $(-l,l)\times \R$.
Furthermore, since the circulation of $B$ around each rectangle $[-r,r]\times[-\rho+2kh,\rho+2kh]$ is given by \eqref{ident-circul-B}, from \eqref{supp-curl-B} we infer that
\begin{equation}\label{circul-B}
    \int_{\gamma} B \in  h(-\mu-\sin\beta,\cos\beta-\eta)\Z,
\end{equation}
for every $1$-rectifiable closed loop $\gamma \subset \R^2\smallsetminus \operatorname{supp}\operatorname{curl}B$.

We now estimate the energy of $B$ in $Q$.
Similarly to the \textit{compression to compression} construction, via a direct computation, we use the fact that $\dist^2(\hat B_i,SO(2))\leq |\hat B_i-Id|^2$ for $i=1,2,3$ to find the following estimates
\begin{equation}\label{bound-energia-B}
		\begin{split}
			&\int_{\Delta_1^\prime}\dist^2(\hat B_1,SO(2))\leq  \frac{h^3}{l}|\mu+\sin\beta|^2+hl|\sin \beta|^2+\frac{h^3}{l}|\eta
			-\cos\beta|^2+hl|1-\cos\beta|^2,\\
			& \int_{\Delta_2^\prime}\dist^2(\hat B_2,SO(2))\leq 4\frac{h^3}{l}|\mu+\sin\beta|^2+hl|\mu+2\sin\beta|^2+4\frac{h^3}{l}|\cos\beta-\eta|^2+hl|2\cos\beta-\eta-1|^2,\\
         &\int_{\Delta_3^\prime}
         \dist^2(\hat B_3,SO(2))\leq C\Big(\frac{h^3}{l}+hl\Big)[(\mu+\sin \beta)^2+(\eta-\cos\beta)^2]\Big|\log\Big(\frac{2r}{l}\Big)\Big|+2hl(\mu^2+(1-\eta)^2).
     \end{split}
 \end{equation}
 The elastic energy in $\Delta^\prime_4$ is in turn given in \eqref{bound-energia-shear}.\\
 To estimate the core energy of $\hat B$ we recall that the support of $\operatorname{Curl}(\hat B)$ is contained in $Q^\prime$. Hence setting $S=Q^\prime$ we have
\begin{equation}\label{bound-core-B}
    |B_{\lambda\eps }(S)|\leq C(r+\lambda\eps)(\rho+\lambda\eps).
\end{equation}

\subsubsection{Rotation to rotation}

We now combine the constructions above in order to obtain the interface between two different rotations $R^-$ and $R^+$.

\begin{proposition}\label{prop-upper-bound}
	For every $R^-,R^+\in SO(2)$ and for every $\eps>0$  there exist $l=l(R^-,R^+),$ $h=h(R^-,R^+)>0$ and a field $A_\eps\in L^1_{\textup{loc}}(\R^2;\R^{2\times2})$
	such that
	\begin{gather*}
		A_\eps(x)=I_{R^-,R^+}(x)\qquad \forall\;x\in \R^2\setminus \left({(-l/2,l/2)\times \R}\right)
	\end{gather*}
	and
	\begin{equation}\label{eq-upper-construction-L}
		F_\eps(A_\eps,Q_L) \leq C L|R^--R^+| \big(\big|\log\big|R^--R^+\big|\big|+1\big)\qquad \forall\; L>h,
	\end{equation}
	where $C>0$ is a dimensional constant. In particular, 
	\begin{equation}\label{upperbound}
		\psi_\infty(R^-,R^+)\leq C|R^--R^+|
		\big(\big|\log\big|R^--R^+\big|\big|+1\big)\qquad \forall\;R^-,R^+\in SO(2).
	\end{equation}
\end{proposition}

\begin{proof}
	Since $F_\eps(I_{R^-,R^+},Q_L)= 2L\lambda$, we need to show \eqref{eq-upper-construction-L} only in the case $|R^--R^+|$ is sufficiently small.
	 Without loss of generality we can assume that $R^-=R_{\theta-\alpha}$ and $R^+=R_{\theta+\alpha}$ for some $\theta \in (0,\pi/2)$ (we are considering the square lattice that has a $\pi/2$ invariance) and small $\alpha\in \mathbb{R}$, so that $|R^--R^+|=2\sqrt{2}\sin\alpha$ and $\sin\alpha<1/8$.
	Set
	\begin{equation}\label{prop-upper-param}
		\begin{split}
			& \eta=\cos\alpha+\cos\theta\sin\theta\sin\alpha,\quad \tilde\eta=\cos\alpha -\cos\theta\sin\theta\sin\alpha,\quad
			\mu=\sin^2\theta\sin \alpha\quad
			\tilde\mu=-\mu,\\
			&h_D=\frac{\eps\tau}{\sin\theta\sin\alpha},\quad
			l_D=\frac{\eps\tau}{\sin\alpha},\quad
			\rho_D=\frac{\eps\tau}{\sin\theta},\quad
			r_D=\eps\tau,\\
			&h_B=h_C=\frac{\eps\tau}{\cos\theta\sin\alpha},\quad
			l_B=\frac{\eps\tau}{\sin\alpha},\quad r_B=r_C=\frac{\eps\tau}{\cos\theta},\quad\rho_B=\rho_C=\eps\tau.
		\end{split}
	\end{equation}
	We apply the \textit{compression to compression} construction with the above choice of the parameters $\eta, \tilde\eta, \mu,\tilde \mu, h_D,l_D, r_D,\rho_D$ to  obtain a matrix field $D_\eps \in L^1_{loc}(\R^2;\R^{2\times2})$ that we will call simply $D$ dropping the dependence in $\eps$ for ease of notation. We have
 \begin{equation}
     S_D:=\operatorname{supp}\operatorname{curl}D\subseteq \bigcup_{k\in \Z}[-r_D,r_D]\times[-\rho_D+2kh_D,\rho_D+2kh_D]
 \end{equation}
 and in view of \eqref{circul-D}   it holds \begin{equation}\label{ident-cur-ruotato-D}
	    \int_\gamma D\in\Z\eps\tau R_{-\theta}e_2,
	\end{equation}
for all closed loop $\gamma\subset \R^2\smallsetminus S_D$.\\
	Consider now the \textit{rotated to compression} construction with parameters $\beta=-\alpha$ and $h_B,l_B,$ $r_B,\rho_B ,\eta, \mu$ as in \eqref{prop-upper-param} to obtain a field $B\in L^1_{loc}(\R^2;\R^{2\times2})$.   We have
 \begin{equation}
S_B:=\operatorname{supp}\operatorname{curl}B\subseteq \bigcup_{k\in \Z}[-r_B,r_B]\times[-\rho_B+2kh_B,\rho_B+2kh_B]
 \end{equation}
 and in view of \eqref{circul-B} we have
	\begin{equation}\label{ident-curl-ruotato-B}
	    \int_\gamma B \in \eps\tau \Z R_{-\theta}e_1,
	\end{equation}
 for all closed loop $\gamma\subset \R^2\smallsetminus S_B$.\\
	The last element we need is a rotated version of the \textit{rotation to compression} construction. Precisely, consider the construction of Section \ref{rotation to compression} with   parameters $\beta=\alpha$ and $ h_C,l_C,r_C,\rho_C,\tilde\eta, \tilde\mu$ as in \eqref{prop-upper-param} to obtain a field $\tilde C\in L^1_{loc}(\R^2;\R^{2\times2})$.
	Now we set
	\begin{gather*}
		C(x_1,x_2)=\tilde C(-x_1,-x_2)  \qquad x\in
		\R^2,
	\end{gather*}
	see Figure \eqref{fig-dislocation-blocks}. Clearly, in view of \eqref{circul-B}, $C$ is such that its circulation satisfies
 \begin{equation}\label{ident-curl-ruotato-C}
	    \int_\gamma C\in \eps\tau \Z R_{-\theta}e_1,
	\end{equation}
 for all closed loop $\gamma\subset \R^2\smallsetminus S_C$ where $S_C=S_B$.\\
	\begin{figure}[t]
		\fontsize{8}{4}{
			\def\svgwidth{370pt}
			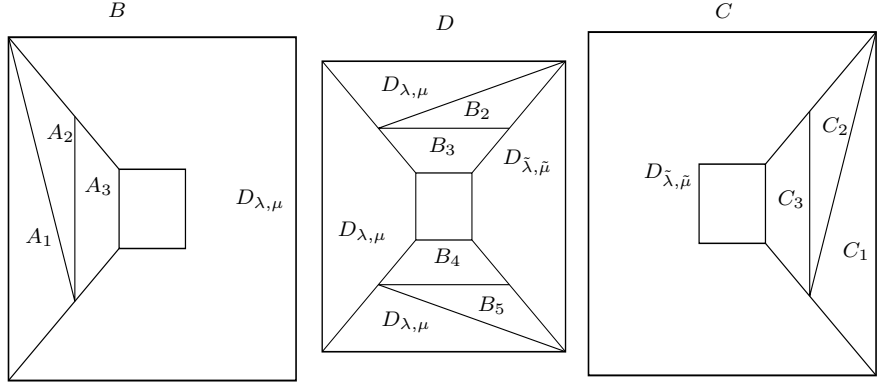
			\caption[]{The matrix fields $B$, $D$, and $C$, simulating horizontal and vertical dislocations.}
			\label{fig-dislocation-blocks}
		}
	\end{figure}
We can now define a field on the whole space by simply gluing together the three constructions, precisely we define the field $A_\theta\in L^1_{\textup{loc}}(\R^2;\R^{2\times 2})$ as follows:
	\begin{equation}
			A_\theta(x)=
			\begin{cases}
			B(x+l_D+l_B) \quad &x\in (-\infty, -l_D)\times \R,\\
			D(x) &x\in (-l_D,l_D)\times \R,\\
			C(x-l_D-l_C) &x\in (l_D,+\infty)\times \R.
			\end{cases}
	\end{equation}
 Observe that,
 because of the boundary conditions imposed to $B,D$ and $C$, and from \eqref{supp-curl-B} and \eqref{supp-curl-D}, no incompatibilities arise in gluing them together. Hence $\operatorname{supp}\operatorname{curl}A_\theta$ is given by the disjoint union of the (appropriately translated) supports of the curl of $B,C$ and $D$:
 \begin{equation}\label{curl-A}
     S_{A_\theta}:=\operatorname{supp}\operatorname{curl}A_\theta\subseteq \big(S_B-(l_D+l_B,0)\big )\cup S_D \cup \big(S_B+(l_D+l_B,0)\big ).
 \end{equation}
 Therefore, if we set $A:=R_\theta A_\theta$, then
 in view of \eqref{ident-cur-ruotato-D}, \eqref{ident-curl-ruotato-B} and    \eqref{ident-curl-ruotato-C}, we have that $A$ satisfies $(H2)$. Furthermore,
	since $R^+=R_{\theta+\alpha}$ and $R^-=R_{\theta-\alpha}$, we have that $A=I_{R^-,R^+}$ in $\R \smallsetminus (-2l_B-l_D,l_D+2l_C)\times \R$.\\
We now prove the upper bound
\eqref{eq-upper-construction-L}.
We set
 \begin{equation}
     h=4h_B+2h_D, \qquad l=4l_B+2l_D,
 \end{equation}
 we observe that $l\leq h$ and estimate the energy of $A$ in the set $ Q_L$, with $L>h$. 	Set $N_D=\lfloor \frac{L}{2h_D}\rfloor=\lfloor \frac{L\sin\alpha\sin\theta}{2\tau\eps}\rfloor$ and $N_B=N_C=\lfloor \frac{L}{2h_B}\rfloor=\lfloor \frac{L\sin\alpha\cos\theta}{2\tau\eps}\rfloor$.
	\begin{figure}[htb]
        \fontsize{8}{4}{
		\def\svgwidth{370pt}
\begingroup%
  \makeatletter%
  \providecommand\color[2][]{%
    \errmessage{(Inkscape) Color is used for the text in Inkscape, but the package 'color.sty' is not loaded}%
    \renewcommand\color[2][]{}%
  }%
  \providecommand\transparent[1]{%
    \errmessage{(Inkscape) Transparency is used (non-zero) for the text in Inkscape, but the package 'transparent.sty' is not loaded}%
    \renewcommand\transparent[1]{}%
  }%
  \providecommand\rotatebox[2]{#2}%
  \newcommand*\fsize{\dimexpr\f@size pt\relax}%
  \newcommand*\lineheight[1]{\fontsize{\fsize}{#1\fsize}\selectfont}%
  \ifx\svgwidth\undefined%
    \setlength{\unitlength}{663.30708661bp}%
    \ifx\svgscale\undefined%
      \relax%
    \else%
      \setlength{\unitlength}{\unitlength * \real{\svgscale}}%
    \fi%
  \else%
    \setlength{\unitlength}{\svgwidth}%
  \fi%
  \global\let\svgwidth\undefined%
  \global\let\svgscale\undefined%
  \makeatother%
  \begin{picture}(1,0.35470085)%
    \lineheight{1}%
    \setlength\tabcolsep{0pt}%
    \put(0,0){\includegraphics[width=\unitlength,page=1]{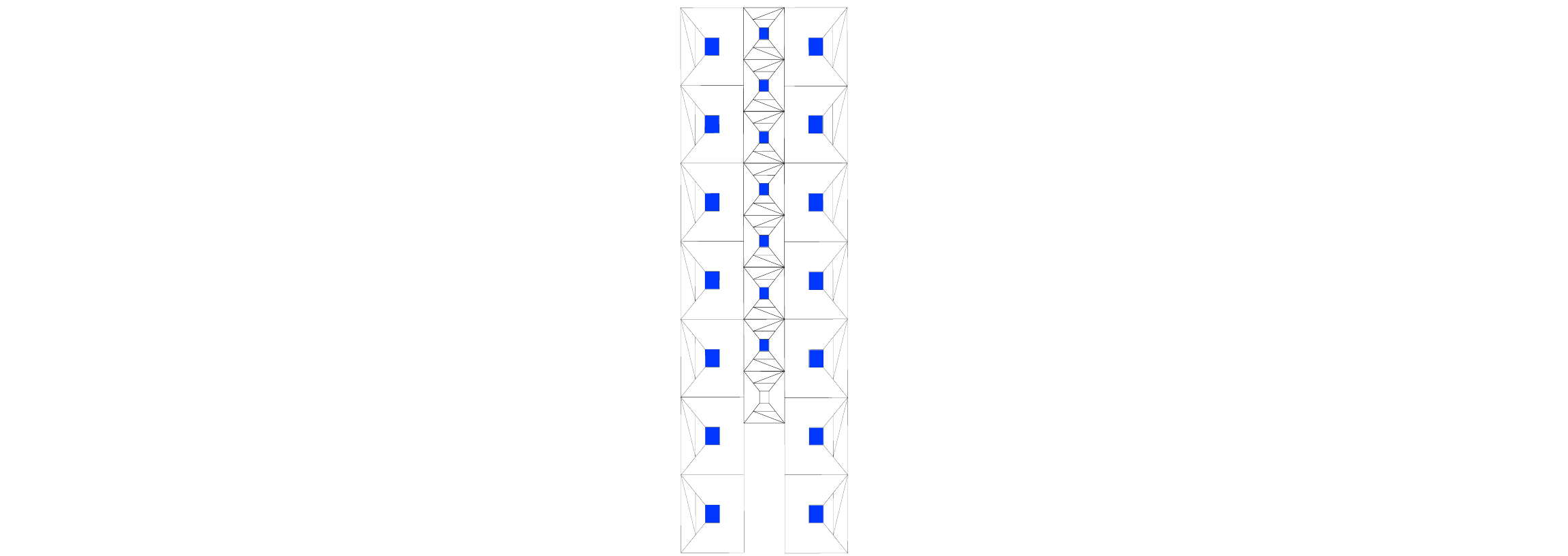}}%
    \put(0.20424435,0.17942178){\makebox(0,0)[lt]{\lineheight{0.40000001}\smash{\begin{tabular}[t]{l}$R_{\theta-\alpha}$\end{tabular}}}}%
    \put(0.68025509,0.17704703){\makebox(0,0)[lt]{\lineheight{0.40000001}\smash{\begin{tabular}[t]{l}$R_{\theta+\alpha}$\end{tabular}}}}%
    \put(0,0){\includegraphics[width=\unitlength,page=2]{anisotropic-construction-big-rectangle3.pdf}}%
  \end{picture}%
\endgroup%

		\caption[]{Construction of the field $A$ in the set  $ Q_L$. The set containing $\operatorname{Curl}A$ is represented in blue.}
             \label{figura4}
             }
	\end{figure}
 From \eqref{bound-core-B}, \eqref{bound-core-D} and \eqref{curl-A}
	we have that $S_A\cap Q_L$ is contained in the set $S$ which is the union of at most  $2N_B+N_D+6$ rectangles of area either $4r_B\rho_B$ or $4r_D\rho_D$, see Figure \ref{figura4}, hence from \eqref{bound-core-B} and \eqref{bound-core-D} we get
	\begin{equation}\label{f-upper-core}
		\begin{split}
			\Ec^\eps(S):=|B_{\lambda\eps}(S)|&\leq  2(N_B+2)\eps^2 (\lambda+\frac{2\tau}{\cos\theta})(\lambda+2\tau)\\
			&+(N_D+2)\eps^2(\lambda+\frac{2\tau}{\sin\theta})(\lambda+2\tau)\\
			&\leq CL\eps \sin\alpha,
		\end{split}
	\end{equation}
	the constant $C$ depending on $\lambda$ and $\tau$.
	Moreover, since the energy is invariant under left multiplication by rotations, we have that the elastic energy of $A_\theta$ is the same of $A$. Hence we found that
	\begin{equation}\label{f-elastic-energy-upper}
		\begin{split}
			\El^\eps( A_\theta,S)&\leq (N_B+2)\int_{Q_{l_B,h_B}\setminus B_{\lambda\eps}(Q_{r_B,\rho_B})} \dist^2(B, SO(2))\, dx\\
			&+(N_C+2)\int_{Q_{l_C,h_C}\setminus B_{\lambda\eps}(Q_{r_C,\rho_C}))} \dist^2(C, SO(2))\, dx\\
			&+(N_D+2)\int_{Q_{l_D,h_D}\setminus B_{\lambda\eps}(Q_{r_D,\rho_D}))} \dist^2(D, SO(2))\, dx.
		\end{split}
	\end{equation}
 We now estimate the elastic energy due to $B$ and $C$. We present only the estimates for $B$, the ones for $C$ being identical.
	Given the choice of the parameters in \eqref{prop-upper-param}, from \eqref{bound-energia-B} and \eqref{bound-energia-shear}
	we infer that
	\begin{equation}
		\begin{split}
			&(N_B+2)\int_{Q_{l_B,h_B}\setminus B_{\lambda\eps}(Q_{r_B,\rho_B})} \dist^2(B, SO(2))\, dx\\
			&\leq C(N_B+2)(h_B\eps\tau\sin\alpha\left(\big|\log(\sin\alpha)\big| +1  \right)+h_B\tau\eps\sin\alpha)\\
&\leq CL\eps\tau\sin\alpha\left(\big|\log(\sin\alpha)\big| +1  \right)+CL\tau\eps\sin\alpha.
		\end{split}
	\end{equation}
	In the same way, from \eqref{bound-energia-D3}, \eqref{bound-energia-D1} and \eqref{bound-energia-shear} we infer that the elastic energy due to $D$ can be estimated as follows
	\begin{equation}
	    \begin{split}
	        &(N_D+2)\int_{Q_{l_D,h_D}\setminus B_{\lambda\eps}(Q_{r_D,\rho_D}))} \dist^2(D, SO(2))\, dx\\
         &\leq CL\eps\tau\sin\alpha\left(\big|\log(\sin\alpha)\big| +1  \right)+CL\tau\eps\sin\alpha.
	    \end{split}
	\end{equation}
	Then, combining the estimates above we obtain
	\begin{equation*}
		\begin{split}
			F_\eps(A, Q_L)&\leq  	CL\sin\alpha\left(\big|\log(\sin\alpha)\big| +1  \right)+CL\sin\alpha\\
			&\leq CL\sin\alpha\left(\big|\log(\sin\alpha)\big| +1  \right),
		\end{split}
	\end{equation*}
	with the constant $C$ depending on $\tau$ and $\lambda$.
	Finally, to infer \eqref{upperbound} we set
	$$
	\tilde A(x)=
	\begin{cases}
		A(x)\quad & x\in Q_{L'},\\
		I_{R^-, R^+}(x) & x\in \R^2\setminus Q_{L'},
	\end{cases}
	$$
	with $L'= h(\lfloor \frac{L}{h}\rfloor-2)$, and notice that $(A,S)\in \mathcal{C}_\eps(Q_L)$ with $S=\supp \Curl A$ and
	\begin{equation}\label{eq-energia-tildeA}
		F_\eps(\tilde A,Q_L)
		\leq CL\sin\alpha \left(|\log(\sin\alpha)| +1  \right)+ C h,
	\end{equation}
	for a dimensional constant $C>0$, thus implying \eqref{upperbound} by taking the limit as $L\to +\infty$.
\end{proof}

\subsection{The optimal grain boundary energy}
Here we give an expression for the asymptotic  energy, namely the $\Gamma$-liminf, per unit  length of a straight interface between two different grains. Again thanks to the invariance under right rotation we only consider the case $n=e_1$.

Given two rotations $R^-,R^+\in SO(2)$ 
we define the energy density
\begin{equation}\label{eq-energy-density}
	\begin{split}
\psi_0(R^-,R^+,Q):=\inf\Big\{\liminf_{\eps\to 0}\frac{1}{\eps}\F_\eps(A_\eps,S_\eps,Q)\,: &(A_\eps,S_\eps)\in \mathcal{C}_\eps(Q)\,,\\ &A_\eps\chi_{B_{\lambda\eps}(S_\eps)}\to 	I_{R^-,R^+} \ \hbox{in } L^1(Q)\Big\}.
	\end{split}
\end{equation}
If $Q=Q_{\frac12}$ is the square of side $1$, then we use the simplified notation
$$
\psi_0(R^-,R^+):=\psi_0(R^-,R^+,Q_{\frac12}).
$$

The main fact leading to the identification of the optimal limiting energy with the thermodynamic limit is the proof that the optimal interfacial energy can be achieved by concentrating defects in a stripe near the interface.

\begin{theorem}\label{th-main-cell-problem}
For every $R^-,R^+\in SO(2)$ there exists a sequence $A_{\eps}\in L^1(Q_{\frac12};\R^{2\times 2})$ such that
\begin{equation}\label{eq-th-main-prop-cell-1}
A_\eps\to I_{R^-,R^+} \quad \textup{in $L^1(Q_{\frac12};\R^{2\times 2})$}\quad\textup{and}\quad\lim_{\eps\to 0^+} F_\eps(A_\eps, Q_{\frac12})
=\psi_0(R^-,R^+),
\end{equation}
and for every $\delta\in (0,\frac14)$
\begin{equation}\label{eq-main-prop-cell-pb}
A_\eps(x)=I_{R^-,R^+}(x)\qquad \forall\;x\in \R^2 \ :\ |x_1|>\delta,
\end{equation}
definitively as $\eps\to0^+$.
\end{theorem}

An immediate consequence of the theorem is the following.

\begin{corollary}\label{cor-main-cell}
The optimal interfacial energy equals the thermodynamic limit
$$
\psi_0(R^-,R^+)=\psi_\infty(R^-,R^+)
\qquad \forall\;R^-,R^+\in SO(2).
$$
\end{corollary}

\begin{proof}

We show separately the two inequalities $\psi_\infty\leq \psi_0$ and $\psi_0\leq \psi_\infty$.

For the first inequality, consider the sequence $A_\eps$, whose existence is guaranteed by Theorem \ref{th-main-cell-problem} such that \eqref{eq-th-main-prop-cell-1} and \eqref{eq-main-prop-cell-pb} hold.
For any  fixed $\delta\in (0,\frac14)$ and $\sigma\in (0,\frac12-\delta)$, we set
\begin{equation}\label{eq-sequence-corollary}
\begin{split}
&\hat A_\eps = A_\eps\chi_{Q_{\frac12-\sigma}}+ I_{R^-,R^+} (1-\chi_{Q_{\frac12-\sigma}}),\\
\hat S_\eps = S_{A_\eps}\cup \{x_1=0,\, &1/2-\sigma\leq |x_2|<1/2\}\cup \{|x_1|\leq \delta, x_2=\pm (1/2-\sigma)\}.
\end{split}
\end{equation}
Since \eqref{eq-main-prop-cell-pb} holds definitively, $(\hat A_\eps, \hat S_\eps)\in \mathcal{C}_\eps(Q_{\frac12})$ (indeed, points which are added to $S_{A_\eps}$ are connected to the boundary of $Q_{\frac12}$ and are not considered to determine the compatibility condition (H2)).
Moreover,
\[
F_\eps(\hat A_\eps, Q_{\frac12}) \leq  F_\eps(A_\eps, Q_{\frac12}) + C(\delta+\sigma),
\]
for a costant $C>0$ depending on $\lambda$. Therefore, by the scaling properties of the energy \eqref{eq-scaling-property-3}, we infer that
\begin{align*}
\eps \psi\big(R^-,R^+,\eps^{-1}\big) & = \inf \Big\{F_\eps (A, Q_{\frac12}) : A = I_{R^-,R^+} \quad\textup{in } \R^2\setminus Q_{{\frac12}-\eta}, \ \eta>0\Big\}\\
& \leq F_\eps (\hat A_\eps,  Q_{\frac12}) \leq  F_\eps(A_\eps, Q_{\frac12}) + C(\delta+\sigma).
\end{align*}
Sending $\eps$ to zero, we get
\begin{align*}
\psi_\infty (R^-,R^+) &= \lim_{\eps\to 0^+}\eps \psi\big(R^-,R^+,\eps^{-1}\big)\\
&\leq
\lim_{\eps\to 0^+}F_\eps(A_\eps, Q_{\frac12}) + C(\delta+\sigma)\\
& = \psi_0(R^-,R^+) + C(\delta+\sigma).
\end{align*}
By the arbitrariness of $\delta,\sigma>0$ the conclusion follows.

As far as the reverse inequality $\psi_0\leq \psi_\infty$ is concerned, for every $\eta>0$ we find $L>0$ and $A_{L}\in L^1_{\textup{loc}}(\R^2;\R^{2\times2})$ be a matrix field such that $A_{L}=I_{R^-,R^+}$ outside $Q_{L'}$, for some $L'<L$, and
\begin{equation}\label{e.scelta}
\psi_\infty(R^-,R^+)\geq
\frac{1}{2L}\psi(R^-,R^+,L) - \eta
\geq \frac{1}{2L}F_1(A_{L},Q_{L})-2\eta.
\end{equation}
We consider the field obtained from $A_L$ extending it periodically in the vertical direction:
\[
\hat A_{L}(x_1,x_2+2kL) := A_{L}(x_1,x_2) \qquad \forall\;(x_1,x_2)\in Q_L,\quad \forall\; k\in \Z.
\]
Then, the fields
\[
A_\eps(x):= \hat A_{L}(\eps^{-1}x) \qquad \forall\; x\in \R^2
\]
belong to $L^1_{\rm loc}(\R^2;\R^{2\times2})$ and, considering the periodicity of $A_\eps$ in the vertical direction with period $2L\eps$, we get
\begin{align*}
\int_{Q_1} |A_\eps(x) -I_{R^-,R^+}(x)|\,d x & \leq \left(\Big\lfloor \frac{1}{2L\eps}\Big\rfloor+2\right) \int_{Q_{L\eps}} |A_\eps(x) -I_{R^-,R^+}(x)|\,d x\\
&=\left(\Big\lfloor \frac{1}{2L\eps}\Big\rfloor+2\right) \eps^2\int_{Q_{L}} |A_{L}(x) -I_{R^-,R^+}(x)|\,d x \to 0 \quad \textup{as }\; \eps \to 0^+,
\end{align*}
where $\Big\lfloor \frac{1}{2L\eps}\Big\rfloor+2$ is an upper bound for the number of disjoint squares of side-length $2L\eps$ with center in the $x_2$-axis contained in $Q_\frac{1}{2}$.
Analogously, the energy of $A_\eps$ is estimated as
\begin{align*}
F_\eps(A_\eps,Q_\frac{1}{2}) \leq \left(\Big\lfloor \frac{1}{2L\eps}\Big\rfloor+2\right)\eps F_1(A_{L},Q_{L}).
\end{align*}
Therefore, using the sequence $A_\eps$ to test $\psi_0$, we get
\begin{align*}
\psi_0(R^-,R^+)&\leq \liminf_{\eps\to 0^+}F_\eps(A_\eps,Q_\frac{1}{2})\\
& \leq \liminf_{\eps\to 0^+}\left(\Big\lfloor \frac{1}{2L\eps}\Big\rfloor+2\right)\eps F_1(A_{L},Q_{L})\\
&=\frac{1}{2L} F_1(A_{L},Q_{L})\stackrel{\eqref{e.scelta}}{\leq} \psi_{\infty}(R^-,R^+)+2\eta,
\end{align*}
and by the arbitrariness of $\eta>0$ the conclusion follows.
\end{proof}

The proof of Theorem \ref{th-main-cell-problem} is postponed to the next sections.

\section{Optimal recovery sequences}
For the proof of Theorem \ref{th-main-cell-problem} we need to construct a recovery sequence for the optimal grain boundary energy which is constantly equal to two rotations on the two sides of the interface.
This will be achieved in several steps in what follows.

\subsection{Energy concentration}
In the first proposition we show that the recovery sequences of matrix fields $A_\eps$ for computing the optimal grain boundary energy concentrate the energy in a neighborhood of the limiting interface.

\begin{proposition}\label{prop-cell-formula-properties}
Let $R^-,R^+\in SO(2)$ and $(A_{\eps},S_\eps)\in \mathcal{C}_\eps(Q_{\frac12})$ be such that
\[
A_\eps\chi_{B_{\lambda\eps}(S_\eps)}\to I_{R^-,R^+} \quad \textup{in $L^1(Q_{\frac12};\R^{2\times 2})$}\quad\textup{and}\quad\lim_{\eps\to 0^+} \frac1\eps \F_\eps(A_\eps,S_\eps,Q_{\frac12})
=\psi_0(R^-,R^+).
\]
Then, for every $\delta \in \big(0,\frac14\big)$, we have that
\begin{equation}\label{eq-prop-striscia}
\lim_{\eps\to 0}\frac{1}{\eps} \F_\eps(A_\eps, S_\eps, Q_{\frac12}\setminus Q_{\delta,\frac12})=0,
\end{equation}
where $Q_{\delta,{\frac12}}=(-\delta,\delta)\times(-\frac12,\frac12)$.
\end{proposition}

\begin{proof}
For $k\in \N\setminus\{0,1\}$ we consider the rectangles
\[
E_{k,i}=\Big(-\frac12,\frac12\Big)\times \Big(-\frac12+\frac{i}{k}, -\frac12+\frac{i+1}{k}\Big) \qquad i=0,..,k-1,
\]
and let $Q^{k,i}$ the square of side-length $\frac{1}{k}$ centered at $(0,-\frac12+\frac{i}{k}+\frac1{2k})$.
For any $\eps>0$ we choose $i_0(\eps)\in \{0,\ldots, k-1\}$ so that
\[
\F_\eps(A_\eps, S_\eps,E_{k,i_0(\eps)})=\min_i \F_\eps(A_\eps, S_\eps, E_{k,i}).
\]
By the scaling property \eqref{eq-scaling-property} we then have that
\begin{equation}
\begin{split}
\frac{1}{k}\,\psi_0(R^-,R^+)&=\lim_{\eps\to 0}\frac{1}{k}\sum_{i=0}^{k-1}\frac1\eps \F_\eps(A_\eps, S_\eps, E_{k,i})\geq\liminf_{\eps\to 0}\frac1\eps \F_\eps(A_\eps, S_\eps, E_{k,i_0(\eps)})\\
&
\geq \liminf_{\eps\to 0}\frac1\eps \F_\eps(A_\eps, S_\eps, Q^{k,i_0(\eps)})\geq \psi_0(R^-,R^+, Q_{\frac{1}{2k}})=\frac{1}{k}\,\psi_0(R^-,R^+).
\end{split}
\end{equation}
This shows also that
\[
\lim_{\eps\to 0} \frac1\eps \F_\eps(A_\eps, S_\eps,E_{k,i})=\lim_{\eps\to 0} \frac1\eps \F_\eps(A_\eps, S_\eps, Q^{k,i})=\psi_0(R^-,R^+,Q_{\frac{1}{2k}}) \qquad \forall\;i\in\{0,...,k-1\}.
\]
Passing to the complementary set of $T_{\delta}$, we conclude \eqref{eq-prop-striscia} for $\delta=\frac{1}{2k}$. The conclusion for every $\delta\leq \frac14$ is an immediate consequence of the monotonicity of $\frac1\eps \F_\eps(A_\eps, S_\eps,Q_{\delta,\frac12})$ with respect to $\delta$.
\end{proof}

\subsection{Clearing out}
If a matrix field is $L^2$ closed to a fixed rotation and its  energy is infinitesimal, then incompatibilities can be actually removed on a large part of the domain replacing the field with the constant rotation. This fact is used to change the  boundary conditions of an optimal sequence for $\psi_0$.

\begin{proposition}\label{th-changing-bc}
Let $R\in SO(2)$ and let $(A_\eps,S_\eps)\in \mathcal{C}_\eps((a,b)\times(c,d))$ and
\begin{equation}\label{eq-th-changing-bc-0}
\frac1\eps\F_\eps(A_\eps,S_\eps,(a,b)\times(c,d))+\int_{(0,1)^2\setminus B_{\lambda\eps}(S_\eps)} |A_\eps-R|^2 \,dx\to 0\quad \hbox{as } \eps\to 0.
\end{equation}
Then, for every $\sigma>0$ there exists a sequence $\tilde A_\eps\in L^2((a,b)\times(c,d);\R^{2\times2})$ converging to $R$ in $L^2$ such that
\begin{equation}\label{eq-changing-bc-2}
\tilde A_\eps=A_\eps\quad \hbox{in } (a,a+\sigma)\times(c,d), \qquad \tilde A_\eps= R\quad \hbox{in }  (a+2\sigma,b)\times(c,d),
\end{equation}
and
\begin{equation}\label{eq-changing-bc-3}
F_\eps(	\tilde A_\eps, (a,b)\times(c,d))\to 0 \quad \hbox{as } \eps\to 0^+.
\end{equation}
\end{proposition}

The proof of Proposition \ref{th-changing-bc} requires the interpolation result below.

\subsubsection{Elastic interpolation with incompatibilities}
In the following, for every $\ell,L>0$ we consider the rectangle of sides $\ell$ and $L$ $$
Q^0_{\ell,L}:=(0,\ell)\times (0,L).
$$
Moreover, given any vector valued function $g\in L^\infty(\R,\R^2)$, we set
\begin{equation}\label{eq-def-elastic-interp}
u(x_1, x_2)=\left(1-\frac{x_1}{\ell}\right)\int_{0}^{x_2} g(t) \,dt\qquad (x_1,x_2)\in Q^0_{\ell, L}.
\end{equation}
Clearly, $u$ is Lipschitz continuous and
\begin{equation}\label{eq-prop-elastic-interp}
\nabla u(x_1, x_2) e_1 = -\frac1\ell \int_{0}^{x_2} g(t) dt,\qquad \nabla u(x_1, x_2) e_2 = \left(1-\frac{x_1}{\ell}\right) g(x_2).
\end{equation}
In particular, $\nabla u e_1$ is a continuous vector field, which extends continuously to $\partial Q^0_{\ell, L}$. Whereas, $\nabla u e_2$ is the tensor product of a continuous function in $x_1$ and a bounded function in $x_2$. Note that, if $g$ is piecewise continuous, the same holds true for $\nabla u e_2$, thus allowing to define the boundary trace of $\nabla u e_2$ (this will be always the case in the applications below).
Moreover, by direct computation
\begin{equation}\label{eq-lemma-elastic-connection-2}
\begin{split}
\int_{Q^0_{\ell,L}} |\nabla u|^2dx &= \frac1\ell\int_{0}^{L}\left|\int_{0}^{x_2} g(t) dt \right|^2dx_2+ \frac{\ell}3 \int_{0}^{L}|g(t)|^2 dt  \\
&\leq  \left(\frac{L^2}{\ell^2}+\frac13\right)\ell \int_{0}^{L}|g(t)|^2 dt .
\end{split}\end{equation}
The discussion above is the building block of the construction of an elastic interpolation with incompatibilities, which will be used in the proof of Proposition \ref{th-changing-bc}.

\begin{lemma}\label{lemma-locally-elastic-connection-eps}
Let $L,\ell>0$ with $\ell\geq \eps$ and $g\in L^2_{loc}(\R)$.  There exists a field $A\in L^2(Q^0_{\ell,L},\R^{2\times 2})$ 
such that
\begin{equation}\label{eq-lemma-traccia}
		Ae_2= \left(1-\frac{x_1}{\ell}\right) g(x_2)
\end{equation}
and $A e_1$ is piecewise continuous and
satisfies
\begin{equation}
Ae_1=
\begin{cases}
	0 &\hbox{on } (0,\ell)\times\{0\},\\
	a_0 &\hbox{on } (0,\ell)\times\{L\},
\end{cases}
\end{equation}
for some $a_0\in \R^2$. Moreover if $\ell\int_{0}^L|g|^2dt<\eps L$, then
\begin{equation}\label{eq-lemma-locally-elastic-connection-1}
\int_{Q^0_{\ell,L}} |A|^2dx  + |B_{\lambda\eps}(S)|\leq  C\eps^{\frac23} L^{\frac23} \left(\ell\int_0^L|g|^2(t) dt\right)^{\frac13} +C\frac{\eps^2L}{\ell},
\end{equation}
where $C>0$ is a dimensional constant, $S:=\supp\Curl A$ is a union of finitely many horizontal segments of lenght $\ell$,
and $Ae_1$ jumps between across these horizontal segments by a vector in $\frac1\ell\eps\tau \Z^2$.

\end{lemma}

\begin{proof}
Without loss of generality, we can assume that $g\not \equiv 0$, otherwise $A=0$ satisfies the conclusions.
We set
\begin{equation}\label{eq-optrimal-N}
M:=\left(\frac{L\eps}{\ell\int_{0}^L |g(t)|^2dt}\right)^{\frac13}.
\end{equation}
We decompose the rectangle $Q^0_{\ell,L}$ into the $N=\lfloor \frac{L}{M\ell}\rfloor+1$ rectangles
$$
E_j=(0,\ell)\times ((j-1)M\ell, jM\ell)\qquad \hbox{with } j=1,\ldots, N-1;\ \ E_N= (0,\ell)\times ((N-1)M\ell, L).
$$

In each of these rectangles $E_i$ we consider the function $u_i$ provided
in \eqref{eq-def-elastic-interp} and define $A_i=\nabla u_i$. By construction the field $A_i$ has a constant tangential component at the horizontal  sides of $E_i$. Precisely for every $i=1,\ldots, N-1$
$$
A_i(x_1, (i-1)M\ell)e_1=0\quad \hbox{and}\quad \ A_i(x_1, iM\ell)e_1=\bar g_i:= -\frac{1}{\ell} \int_{(i-1)M\ell}^{iM\ell} g(t)\, dt.
$$
We then project on the lattice $\tau\eps\Z^2$ the vector $\ell \bar g_i$ and get
$$
g_i^\eps:= \frac{P_{\tau\eps \Z^2}(\ell \bar g_i)}{\ell},
$$
where $P_{\eps\tau\Z^2}(w)= (i\eps\tau,j\eps\tau)$  for all $w\in \R^2$ with $w\in [i\eps\tau, (i+1)\eps\tau)\times [j\eps\tau,(j+1)\eps\tau)$ and $(i,j)\in \Z^2$. So that in particular $|\ell (g_i^\eps-\bar g_i)|<\sqrt2\tau\eps$. Then we define the field as follows
$$
A(x)= A_1\chi_{E_1}(x) +\sum_{i=2}^N\chi_{E_i}(x) \left[A_i(x) +(g_{i-1}^\eps-\bar g_{i-1})\otimes e_1\right].
$$

Note that the field $Ae_1$ jumps between two neighbouring  rectangles by a vector in $\frac1\ell\eps\tau \Z^2$ and
\begin{equation}\label{e.rotore}
\supp\Curl A = \bigcup_{j\in J} (0,\ell)\times  \{jM\ell\}  ,\qquad \hbox{with } J\subset \{1,\ldots, N-1\}.
\end{equation}
 It remains to compute the energy of $A$.
By direct computation we have that
\begin{equation}
\begin{split}
\int_{Q^0_{\ell,L}} |A|^2dx& \leq 2\sum_{j=1}^N\left(\int_{E_j}|\nabla u_j|^2 dx +  M \ell^2 |g_j^\eps-\bar g_j|^2\right)\\
&\leq2\left(M^2+\frac13\right)\ell\sum_{j=1}^N\int_{(j-1)M\ell}^{jM\ell}|g(t)|^2 dt + 2 \sum_{j=1}^NM \ell^2 |g_j^\eps-\bar g_j|^2\\
&\leq 2\left(M^2+\frac13\right)\ell \int_{0}^{L}|g(t)|^2 dt +C \frac{L}{\ell} \eps^2 .
\end{split}
\end{equation}

Combining this estimate with the information on the support $S$ of the $\Curl A$ in \eqref{e.rotore} and the definition of $M$ we get
\begin{equation}\label{eq-connection-mixed}
\begin{split}
\int_{Q^0_{\ell,L}}|A|^2 dx +|B_{\lambda\eps}(S)|&\leq
C\,\frac{L}{\ell}\eps^2 +
CM^2\ell\,\int_{0}^{L}|g(t)|^2 dt +
(N-1)\big(\lambda\eps \ell+\pi\lambda^2\eps^2\big)\\
&\leq C
\frac{L}{\ell}\eps^2 +
CM^2\ell\,\int_{0}^{L}|g(t)|^2 dt +
C \frac{L}{M}\eps +C \frac{L\eps^2}{M\ell}\\
&\leq C \frac{L}{\ell}\eps^2+
C\,\eps^{\frac23}L^{\frac23}\left(\ell\,\int_{0}^{L}|g(t)|^2 dt\right)^{\frac13},
\end{split}
\end{equation}
where we used that $\eps\leq \ell$ to bound the last term.
\end{proof}

\subsubsection{Proof of Proposition \ref{th-changing-bc}}

Without loss of generality we prove the result in the case $(a,b)\times(c,d)=(0,1)^2$.

Fix $\sigma>0$ and denote  $T^+_\sigma=(\sigma,2\sigma)\times (0,1)$. We then set
\begin{equation}\label{eq-zero energy}
	\omega_\eps:=\frac{1}{\eps}	\int_{T^+_\sigma\setminus B_{\lambda\eps}(S_\eps)}\dist^2(A_\eps,SO(2))dx +\frac{1}{\eps}	|B_{\lambda\eps}(S_\eps)|+\int_{T^+_\sigma\setminus B_{\lambda\eps}(S_\eps)} |A_\eps-R|^2 \,dx,
\end{equation}
and using \eqref{eq-th-changing-bc-0}, we have that $\omega_\eps\to 0$ as $\eps\to 0^+$.
Let  $\rho_\eps:=\eps M_\eps$, with $M_\eps=1/\sqrt{\eps\vee\omega_\eps}$, and  $N_\eps:=\lfloor\sigma/\rho_\eps\rfloor$. Note that  $\rho_\eps\to 0$ and $\omega_\eps M_\eps\to 0$.

We consider the $N_\eps$ rectangles $T_i=(\sigma+i\rho_\eps, \sigma+(i+1)\rho_\eps)\times (0,1)\subset T^+_\sigma$, with $i=0,\ldots , N_\eps-1$.
Since from \eqref{eq-zero energy} we have that
$$
\sum_{i=0}^{N_\eps-1} \frac1\eps \F_\eps(A_\eps, S_\eps, T_i)+\int_{T_i\setminus B_{\lambda\eps}(S_\eps)}|A_\eps-R|^2\, dx\leq \omega_\eps,
$$
we can find an index $i_0\in \{0,\ldots, N_\eps-1\}$ such that
\begin{equation}\label{eq-scelta-striscia}
	\frac1\eps \F_\eps(A_\eps, S_\eps, T_{i_0})+\int_{T_{i_0}\setminus B_{\lambda\eps}(S_\eps)}|A_\eps-R|^2\, dx\leq \frac{1}{N_\eps}\omega_\eps=\frac1\sigma \omega_\eps \rho_\eps.
\end{equation}
In particular, from \eqref{eq-scelta-striscia}, we infer that
\[
|B_{\lambda\eps}(S_{\eps})\cap T_{i_0}|\leq C\eps\, \omega_\eps \rho_\eps,
\]
and a simple covering argument implies that for $\eps$ small enough
\[
|B_{\lambda\eps}(p(S_{\eps}))\cap (\sigma+i_0\rho_\eps,\sigma+ (i_0+1)\rho_\eps)|\leq C \omega_\eps \rho_\eps \leq \frac{\rho_\eps}{4},
\]
where $p(x_1, x_2)=x_1$ is the projection on the first component and the measure of the tubular neighbourhood of the projection is the one dimensional Lebesgue measure.
Therefore, we can choose a vertical section $\gamma^-=\{t_0\}\times (0,1)$ for some $t_0\in (\sigma+i_0\rho_\eps,\sigma+i_0\rho_\eps +\rho_\eps/2)$ such that
\begin{equation}\label{eq-sezione-stima-2}
	\gamma^-\cap B_{\lambda\eps}(S_{\eps}) = \emptyset,
\end{equation}
and
\begin{equation}\label{eq-sezione-stima}
	\rho_\eps\int_{\gamma^-} |A_\eps(x)-R|^2 d\mathcal H^1(x)\le 2\int_{T_{i_0}\setminus B_{\lambda\eps}(S_\eps)} |A_\eps(x)-R|^2 dx \leq C \omega_\eps\rho_\eps= C\eps \omega_\eps M_\eps.
\end{equation}

Moreover we can choose $\gamma^-$ so that the tangential component  of $A_\eps$ on $\gamma^-$ is well defined in $L^2(\gamma^-)$.
We apply then Lemma \ref{lemma-locally-elastic-connection-eps} in the rectangle
$\hat T:=(t_0, \sigma+(i_0+1)\rho_\eps)\times (0,1)$ with
\[
g(t) = \big(A_\eps(t_0,t) - R\big)e_2,
\]
getting a  field $\hat A_\eps: [t_0, \sigma+(i_0+1)\rho_\eps]\times (0,1)  \to \R^{2\times 2}$, with $\hat S_\eps:=\supp\curl \hat A_\eps$, satisfying
\begin{equation}\label{eq-lemma-locally-elastic-connection-1-1}
\begin{split}
\int_{\hat T} |\hat A_\eps|^2dx  + |B_{\lambda\eps}(\hat S_\eps)|&\leq
C\eps^{\frac23} \left(\rho_\eps\int_{0}^1|g|^2 dt\right)^{\frac13} +C\frac{\eps^2}{\rho_\eps}\\
&\leq C\eps^{\frac23}  \left(\eps \omega_\eps M_\eps\right)^{\frac13} +\frac{\eps}{M_\eps}= C\eps((\omega_\eps M_\eps)^{\frac13} +\frac{1}{M_\eps})\leq C\eps(\omega_\eps\vee \eps)^{\frac16}
\end{split}
\end{equation}
where we used  \eqref{eq-sezione-stima} and the definition of $\rho_\eps$ and $M_\eps$, and such that
the field
\begin{equation}
\tilde A_\eps(x)=
\begin{cases}
A_\eps(x)\qquad & \hbox{if } x\in (0,t_0)\times(0,1)\\
\hat A_\eps(x) +R & \hbox{if } x\in [t_0, \sigma+(i_0+1)\rho_\eps)\times (0,1)\\
R & \hbox{if } x\in [\sigma+(i_0+1)\rho_\eps,1)\times (0,1).
\end{cases}
\end{equation}
does not have curl along $\gamma^-$ and $\{\sigma+(i_0+1)\rho_\eps\}\times (0,1)$, thanks to \eqref{eq-lemma-traccia}.
Setting $\tilde S_\eps:= S_\eps\cap [(0,t_0)\times(0,1)]\cup\hat S_\eps$, and using Lemma \ref{lemma-locally-elastic-connection-eps} and \eqref{eq-sezione-stima-2} we also infer that the field $\tilde A_\eps$ is admissible and satisfies
\begin{equation*}
\frac1\eps\F_\eps(\tilde A_\eps,\tilde S_\eps; (0,1)^2)\leq \frac1\eps\F_\eps(A_\eps,S_\eps; (0,t_0)\times(0,1)) +\frac{1}{\eps}\int_{\hat T} |\hat A_\eps|^2dx  +\frac1\eps |B_{\lambda\eps}(\hat S_\eps)| \leq C(\omega_\eps\vee \eps)^{\frac16},	
\end{equation*}
which concludes the proof.

\subsection{Proof of Theorem \ref{th-main-cell-problem}}

	The proof now is a direct application of Proposition \ref{th-changing-bc}. Indeed, for a fixed $\delta>0$, from Propositions \ref{prop-cell-formula-properties}
and the compactness result (Theorem \ref{t.compactness}) we  obtain an optimal sequence $(A_\eps,S_\eps)\in \mathcal{C}_\eps(Q_\frac12)$ such that
\begin{equation}\label{eq-th-conclusione}
A_\eps\chi_{Q_1\setminus B_{\lambda\eps}(S_\eps)}\to I_{R^-,R^+} \quad \textup{in $L^2(Q_1;\R^{2\times 2})$}\quad\textup{and}\quad\lim_{\eps\to 0^+} \frac1\eps\F_\eps( A_\eps,S_\eps;Q_{\delta,\frac12})
=\psi_0(R^-,R^+),
\end{equation}
with $Q_{\delta,\frac12} = (-\delta, \delta)\times (-1/2,1/2)$ and
	\begin{equation}
		\frac1\eps\F_\eps(A_\eps, S_\eps,Q_\frac12\setminus Q_{\delta,\frac12})\to 0.
	\end{equation}
Therefore we can apply Proposition \ref{th-changing-bc} to the rectangles $(\delta, 1/2)\times (-1/2,1/2)$ and $(-1/2,-\delta]\times (-1/2,1/2)$ (with the constant rotations $R^-$ and $R^+$ respectively), and obtain a new admissible field $\tilde A_\eps$ for which satisfies \eqref{eq-th-conclusione} and
$$
\tilde A_\eps = I_{R^-,R^+} \quad \hbox{in } Q_\frac12\setminus Q_{\delta,\frac12}.
$$
Extending $\tilde A_\eps$ with $I_{R^-,R^+}$ outside $Q_\frac12$ provides
\eqref{eq-main-prop-cell-pb} and conclude the proof.

	\section{$\Gamma$-convergence}\label{gamma coonvergence}

We are now ready to prove the main result of the paper, namely the $\Gamma$-convergence result of the functionals $F_\eps$.
For every $R^-,R^+\in SO(2)$, $n\in \mathbb{S}^1$, let $R_n\in SO(2)$ be the rotation such that $R_ne_1=n$ we define
\begin{equation}
	\label{def-Phi}
	\Phi(R^-,R^+,n):= \psi_0(R^-R_n,R^+R_n)=\psi_\infty(R^-R_n,R^+R_n).
\end{equation}

\begin{remark}\label{rem-properties-psi0}
	From the invariance  properties of the energy (see Remark \ref{rem-invariance-rotation0} and \ref{rem-invariance-rotation}) and the definitions of $\psi_\infty$ we deduce that, for all $R^-,R^+\in SO(2)$ and $n\in\mathbb{S}^1$, the energy density $\Phi$ satisfies
	\begin{equation}\label{eq-cell-problem-eps-1}
		\begin{split}
			\Phi(R^-,R^+, n)=\lim_{L\to +\infty}\frac{1}{2L}\inf\Big\{F_1(A,Q_L^n)\,:\ &A\in L^1_{loc}(\R^2;\R^{2\times 2})\,,\ \\
			&A = I_{R^-,R^+,n} \hbox{ in } \R^2\setminus Q^n_{L'}, \hbox{ for some } L'<L\Big\},
		\end{split}
	\end{equation}
	where
	\begin{gather*}
		I_{R^-,R^+,n}(x):=R^-\chi_{\{x\cdot n<0\}}(x) +R^+\chi_{\{x\cdot n\geq0\}}(x) ,
	\end{gather*}
and the set $Q_L^n$ is the square  centered at $0$,  with a side parallel to $n$ and length $2L$.
\end{remark}

We now prove the $\Gamma$-convergence result.

\begin{theorem}\label{t.main1}
The functionals $F_\eps$ $\Gamma$-converge to
\begin{equation}\label{eq-limit}
F_0(A)=\begin{cases}{\displaystyle
\int_{J_A}\Phi(A^-,A^+,\nu_A) \,d\Ha^{1}},\qquad & \hbox{ if $A$ is a microtation},\\
+\infty & otherwise,
\end{cases}
\end{equation}
with respect to the convergence of $A_\eps$ to $A$ given by
\[
A_\eps \chi_{\Omega\setminus B_{\lambda\eps}(S_{A_\eps})}\to A \qquad \textup{in } L^1(\Omega;\R^{2\times 2}).
\]
The energy density $\Phi:SO(2)\times SO(2)\times \mathbb{S}^1\to [0,+\infty)$ satisfies the Read-Shockley bounds
\begin{equation}\label{eq-log-bounds}
C_1|R^--R^+| (|\log|R^--R^+||+1)\leq \Phi(R^-,R^+,n)\leq C_2|R^--R^+| (|\log|R^--R^+||+1)
\end{equation}
 for every $R^-,R^+\in SO(2)$ and $ n\in \mathbb{S}^1$, with $C_1,C_2>0$ dimensional constants.
\end{theorem}

\begin{remark}\label{rem-gamma-convergence}
We recall that the $\Gamma$-convergence consists in the following two statements:
\begin{itemize}
\item[i)] \emph{$\Gamma$-$\liminf$ inequality}: For every
$A_\eps\chi_{\Omega\setminus B_{\lambda\eps}(S_{A_\eps})} \to A$ micro-rotation in $L^1(\Omega;\R^{2\times 2})$  we have
\begin{equation}\label{eq-lower}
F_0(A)\leq \liminf_{\eps\to 0}F_\eps(A_\eps);
\end{equation}
\item[ii)]   \emph{$\Gamma$-$\limsup$ inequality}: For every micro-rotation $A$ there exists a sequence $A_\eps$ such that $A_\eps\chi_{\Omega\setminus B_{\lambda\eps}(S_{A_\eps})}\to A$  in $L^1(\Omega;\R^{2\times2})$ and
\begin{equation}\label{eq-upper}
F_0(A)\geq \limsup_{\eps\to 0}F_\eps(A_\eps).
\end{equation}
\end{itemize}
\end{remark}

\begin{proof}
We start showing the
\textit{$\Gamma$-liminf inequality}: Let $A_\eps\in L^1(\Omega;\R^{2\times 2})$ and $A$ be a micro-rotation, with $A_\eps\chi_{\Omega\setminus B_{\lambda\eps}(S_{A_\eps})} \to A$.
Without loss of generality we can assume that $F_\eps(A_\eps,\Omega)$ is bounded and consider the Radon measures
\begin{equation}
			\label{eq-lower-measure}
			\mu_\eps(E):= F_\eps (A_\eps, E)   \qquad E\subset \Omega,
		\end{equation}
which we may assume converging weakly to some Radon measure $\mu$.
Let
\[
\sigma:=|A^+-A^-| \mathcal{H}^1\LL J_A, \qquad \mu=:\mu^a+\mu^s=:f\sigma+\mu^s,
\]
where $\mu^a$ and $\mu^s$ are the absolutely continuous part and the singular part of $\mu$ with respect to $\sigma$, respectively.
Then,  for $\sigma$-almost every $x_0$ in $\Omega$, and therefore for $\Ha^1$-almost every $x_0$ in $J_A$, we have that
\begin{equation}
\begin{split}\label{eq-lower-bound-blow-up1}
				f(x_0)&=\lim_{\rho\to 0} \frac{\mu(Q_\rho(x_0))}{\sigma(Q_{\rho}(x_0))}=
				\lim_{\rho\to 0}\frac{2\rho}{\sigma(Q_\rho(x_0))}\frac{\mu(Q_\rho(x_0))}{2\rho}\\
				&
				=\frac{1}{|A^-(x_0)-A^+(x_0)|}\lim_{\rho\to 0}\frac{\mu(Q_\rho(x_0))}{2\rho}.
			\end{split}
		\end{equation}
In order to prove the $\Gamma$-$\liminf$ inequality is enough to show
		\begin{equation}
			\label{eq-lower-derive}
			\lim_{\rho\to 0} \frac{\mu(Q_\rho(x_0))}{2\rho}\geq \Phi(A^-(x_0),A^+(x_0),\nu_A(x_0))\quad \hbox{for $\Ha^1$-a.e. } x_0\in J_A.
		\end{equation}
Indeed, from \eqref{eq-lower-derive} combined with \eqref{eq-lower-bound-blow-up1}, we infer
$$
		f(x_0)\geq \frac{\Phi(A^-(x_0),A^+(x_0),\nu_A(x_0))}{|A^+(x_0)-A^-(x_0)|} \quad \hbox{for $\Ha^1$-a.e. } x_0\in J_A,
$$
and, as a  consequence,
		\begin{equation*}
			\lim_{\eps\to 0} F_\eps(A_\eps)\geq \mu(\Omega)\geq\mu^a(\Omega)= \int_\Omega f d\sigma\geq \int_{J_A}\Phi(A^-(x),A^+(x),\nu_A(x))d \mathcal{H}^1(x),
		\end{equation*}
		which is the required lower bound.

\medskip

It remains to show \eqref{eq-lower-derive} for all $x_0$ such that \eqref{eq-lower-bound-blow-up1} holds true.

Thanks to the invariance under right rotation of the energy it is not restrictive to assume $\nu_A(x_0)=e_1$.

In order to simplify the notation we  write $A^+$ and $A^-$ instead of $A^+(x_0)$ and $A^-(x_0)$. We can assume that $x_0$ also satisfies
		\begin{equation}\label{eq-blow-up-0}
			\lim_{\rho\to 0} \mean_{Q_\rho(x_0)}|A(x)-I_{A^+,A^-}(x-x_0)|dx=0.
		\end{equation}
In other words, \eqref{eq-blow-up-0} accounts for
\[
A^\rho(y):=A(2\rho y+x_0)\to I_{A^+,A^-}\qquad L^1(Q_\frac12;\R^{2\times 2})\quad\textup{as $\rho\to 0$}.
\]
Since  $\mu(Q_\rho(x_0))=\lim_{\eps\to 0} \mu_\eps(Q_\rho(x_0))$ for a.e. $\rho>0$ and $A_\eps\chi_{\Omega\setminus B_{\lambda\eps}(S_{A_\eps})}\to A$, we have that
		\begin{equation}
			\label{eq-blow-up-1}
			\liminf_{\rho\to 0}\lim_{\eps\to 0}\left(\mean_{Q_{\rho}(x_0)}|A_\eps\chi_{\Omega\setminus B_{\lambda\eps}(S_{A_\eps})}-A|\,dx +  \frac{\left|\mu_\eps(Q_\rho(x_0))- \mu(Q_\rho(x_0))\right|}{2\rho} \right)= 0.
		\end{equation}
		Therefore, we can construct a diagonal sequence $\rho_\eps$ such that
		\begin{equation}
			\label{eq-blow-up-1bis}
			\lim_{\eps\to 0}\left(\mean_{Q_{\rho_\eps}(x_0)}|A_\eps\chi_{\Omega\setminus B_{\lambda\eps}(S_{A_\eps})}-A|\,dx +  \frac{\left|\mu_\eps(Q_{\rho_\eps}(x_0))- \mu(Q_{\rho_\eps}(x_0))\right|}{2\rho_\eps}\right)= 0.
		\end{equation}
Observe that from \eqref{eq-blow-up-1bis} and \eqref{eq-blow-up-0} we get
		\begin{equation}
			\label{eq-blow-up-2}
			\begin{split}
				\lim_{\eps\to 0}\mean_{Q_{\rho_\eps}(x_0)}|A_\eps\chi_{\Omega\setminus B_{\lambda\eps}(S_{A_\eps})}(x)-I_{A^+,A^-}(x-x_0)|\,dx=0.
			\end{split}
		\end{equation}
This implies that  $\tilde A_\eps(y)=A_\eps\chi_{\Omega\setminus B_{\lambda\eps}(S_{A_\eps})}(2\rho_\eps y+x_0)\to I_{A^+,A^-}$ in $L^1$. Moreover, from \eqref{eq-blow-up-1bis} and \eqref{eq-lower-bound-blow-up1} we infer
\begin{align*}
			\lim_{\rho\to 0}\frac{\mu(Q_\rho(x_0))}{2\rho}&=\lim_{\eps\to 0}\frac{\mu_\eps(Q_{\rho_\eps}(x_0))}{2\rho_\eps}=\lim_{\eps\to 0}\frac{F_\eps(A_\eps,Q_{\rho_\eps}(x_0))}{2\rho_\eps}\\
			&=\lim_{\eps\to 0} F_{\eps/2\rho_\eps}(\tilde A_\eps, Q_\frac12)\geq \psi_0(A^-,A^+)=\Phi(A^-,A^+,e_1),
\end{align*}
which is \eqref{eq-lower-derive}, recalling that we assumed $\nu_A(x_0)=e_1$. So that the proof of the lower bound is complete.

\medskip

We now show the \emph{$\Gamma$-limsup inequality}: Fix a micro-rotation $A$.
By a density argument (see \cite{Bellettini-Chambolle-Goldman}) we can assume that $A$ is piecewise constant on a finite partition of $\Omega$ with polygons. The construction of the recovery sequence is then an immediate consequence of the characterization of the line tension by means of the asymptotic formula.
Let
		$$
		A=\sum_{i=1}^N R_i\chi_{E_i},
		$$
with $R_i\in SO(2)$ and $\{E_i\}$ is a polyhedral partition of $\Omega$. The boundary of the partition is made of segments $\Gamma_{i,j}=\overline E_i\cap\overline E_j\neq\emptyset$, with $i,j\in\{1,\ldots, N\}$, which have normal $\nu_{ij}$, and that can meet in a finite set of points $\{x_1,\ldots, x_L\}$.
With this notation we have
		$$
		F_0(A)=\sum_{i,j}\mathcal{H}^1(\Gamma_{ij})\Phi(R_i,R_j,\nu_{i,j}).
		$$
		\begin{figure}[htb]
			\def\svgwidth{150pt}	\fontsize{10}{4}{
			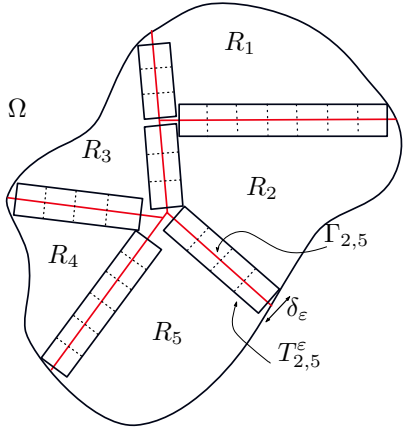
			\caption[]{Construction of the field $A_\eps$ for the $\Gamma-\limsup$}
			\label{figura5}}
		\end{figure}
Then, we consider a sequence $\delta_\eps$, with $\eps<\!\!<\delta_\eps<\!\!<1$, and we  construct rectangles $T^\eps_{ij}$ with one side of length $\delta_\eps$ orthogonal to the segment $\Gamma_{ij}$, symmetric with respect to $\Gamma_{ij}$ (as in Figure \ref{figura5}) and height such that the rectangles $T^\eps_{ij}$ are pairwise disjoint and
\[
\bigcup_{ij}\Gamma_{ij}\subset \bigcup_{i,j} T^\eps_{ij}\cup \bigcup_{i=1}^{L} B_{2\delta_\eps}(x_i).
\]
Now  from \eqref{eq-cell-problem-eps-1} for each pair $i,j$ we find a field $A_\eps^{i,j}$ such that
\begin{equation}\label{eq-ciccio}
		\lim_{\eps\to 0} \frac{\eps}{2\delta_\eps}F_1(A_\eps^{i,j}, Q_{\frac{\delta_\eps}\eps}^{\nu_{ij}})=\Phi(R_i,R_j,\nu_{ij}),
\end{equation}
and
		$$
A_\eps^{i,j}=I_{R_i,R_j,\nu_{ij}} \quad \hbox{in  } \R^2\setminus Q_{L^{i,j}_\eps}^{\nu_{ij}}\ \hbox{for } L^{i,j}_\eps<\frac{\delta_\eps}\eps.
$$
Set
\begin{equation}\label{eq-AL-estesa1}
	\hat A_\eps^{i,j}(x+2k\delta_\eps\nu_{ij}^\perp)=A_\eps^{i,j}\Big(\frac{x}{\eps}\Big)\qquad \forall x\in
	\big\{x\in \R^2:\ |x\cdot\nu_{ij}^\perp|\leq \delta_\eps\big\} , \quad \forall\:\, k\in\Z.
\end{equation}
We then  define the field
$$
A_\eps(x)=\begin{cases}
\hat A_\eps^{i,j}(x-p_{ij}) \quad &\hbox{if } x\in T^\eps_{ij}\\
A(x) & \hbox{otherwise in } \Omega,
\end{cases}
$$
where $p_{ij}\in \Gamma_{ij}$ are given points.
Note that  $T^\eps_{ij}-p_{ij}\subset \{x\in \R^2:\ |x\cdot\nu_{ij}|\leq \delta_\eps\}$.
Since each $T^\eps_{ij}$ can be covered with $N_{ij}^\eps$ disjoint copies of cubes of the type $Q_{\delta_\eps}^{\nu_{ij}}$ with
\begin{equation}\label{eq-franco}
2N_{ij}^\eps\delta_\eps\to \mathcal{H}^1(\Gamma_{ij}),
\end{equation}
when $\eps\to 0$ (and hence $\delta_\eps\to 0$). We have that $A_\eps\to A$ and
		\begin{equation}
			F_\eps(A_\eps,\Omega)\leq \sum_{i,j} N_{ij}^\eps F_\eps(\hat A_\eps^{i,j}, Q_{\delta_\eps}^{\nu_{ij}})+ C M \delta_\eps\leq \sum_{i,j} N_{ij}^\eps \eps F_1(A_\eps^{i,j}, Q_{\frac{\delta_\eps}\eps}^{\nu_{ij}})		+ C M \delta_\eps,
		\end{equation}
where $M$ is the number of segments $\Gamma_{ij}$. Taking the limit as $\eps\to 0$, and using \eqref{eq-ciccio} e \eqref{eq-franco} we conclude \eqref{eq-upper}.

\medskip
It remains to show the properties of the energy density $\Phi$.
The bound \eqref{eq-log-bounds} is an immediate consequence of Propositions \ref{prop-lower-bound} (the lower bound) and Proposition \ref{prop-upper-bound} (the upper bound), together with the fact that $\Phi(R^-,R^+,n)=\psi_\infty(R^-R_n,R^+R_n)$, for all $R^-,R^+\in SO(2)$, $n\in\mathbb{S}^1$, and $R_ne_1=n$. In particular, \eqref{eq-log-bounds} implies that $\Phi(R^-,R^+,n)$ tends to zero with infinite slope when $|R^+-R^-|$ tends to zero.
Finally, since $F_0$ is a $\Gamma$-limit, it is lower semicontinuous with respect to the $L^1$ convergence. This fact easily implies that $\Phi$ is lower semicontinuous and subadditive in the first two entries.
\end{proof}

\begin{remark}
Note that the recovery sequence constructed in the proof of the $\Gamma$-$\limsup$ inequality converges $L^1(\Omega;\R^{2\times 2})$ to the limiting configuration $A$, not only in the topology of the statement (i.e., the convergence of $A_\eps \chi_{\Omega\setminus S_{A_\eps}}\to A$).
\end{remark}

	\appendix

	\section{Coverings}\label{s:covering}

	The following elementary covering argument is used in several steps and we report it here for readers' convenience (see \cite{LL2}).

	\begin{lemma}\label{preliminare}
		Let $S\subset\R^n$ be a compact set, let $\sigma>0$ and $h\in \R$, $h>1$. Then there exists a constant $C(h,n)>0$ (depending only on $h$ and $n$) such that
		$$
		|B_{h\sigma}(S)|\leq C(h,n) |B_\sigma(S)|\,.
		$$
	\end{lemma}

	\begin{proof}
		Let $\{x_j\}_{j\in J} \subset S$ be a maximal set such that $|x_i-x_j|\geq 2\sigma$ for every $i,j\in J$, $i\neq j$ (maximal for inclusion: for every $x\in S$ there exists $j\in J$ such that $|x_j-x|<2\sigma$).
		Note that, since $S$ is bounded, the set $J$ is finite.
		By construction, for every point in $B_{h\sigma(S)}$ there exists $j\in J$ such that
		\[
		|x-x_j|\leq \dist(x,S)+2\sigma\leq (h+2)\sigma,
		\]
		i.e.,
		$$
		B_{h\sigma}(S)\subseteq \bigcup_{j\in J} B_{(h+2)\sigma}(x_j)\,.
		$$
		Then, using the fact that the balls $B_{\sigma}(x_j)\cap B_{\sigma}(x_i)=\emptyset$ for every $i,j\in J$, $i\neq j$, we get that
		\begin{align*}
			|B_{h\sigma}(S)|&\leq \sum_{j\in J}|B_{(h+2)\sigma}(x_j)|= (h+2)^n\sum_{j\in J}|B_{\sigma}(x_j)|\\
			&=(h+2)^n|\bigcup_{j\in J} B_{\sigma}(x_j)|\leq (h+2)^n |B_\sigma(S)|.
		\end{align*}
	\end{proof}


	\begin{lemma}\label{l.covering1}
		Let $\mu$ be a measure whose support is contained in the union of balls $\{B_{\rho_i}(x_i)\}_{i\in I}$
		with equi-bounded radii $\rho_i\leq r$ for some $r>0$ for every $i\in I$.
		Then, there exists a constant $\eps_0>0$ with this property:
		for every $R\geq2r$ there exists a subfamily $I'\subset I$ such that
		\begin{itemize}
			\item the balls $\{B_{2R}(x_i)\}_{i\in I'}$ are disjoint,
			\item  $\mu (\bigcup_{i\in I}B_{R}(x_i)) \geq \eps_0 \mu(\R^n)$.
		\end{itemize}
	\end{lemma}

	\begin{proof}
		Let $\{x_i\}_{i\in J}$ be a maximal set of centers such that $|x_i-x_j|\geq R/2$ for every $i,j\in J$.
		It then follows that
		\[
	\supp(\mu) \subset \bigcup_{i\in I}B_{\rho_i}(x_i)\subset \bigcup_{i\in J}B_{R}(x_i).
		\]
		Indeed, if $j\in I\setminus J$, then there exists $i \in J$ such that $x_j\in B_{R/2}(x_i)$ and therefore
		$B_{\rho_j}(x_j) \subset B_{R/2+\rho_j}(x_i)\subset B_{R}(x_i)$ (because $\rho_j\leq R/2$).

		By the choice of the set $J$, there exists a dimensional constant $N\in \N$ such that
		each $x\in \R^n$ belongs to at most $N-1$ balls $B_{4R}(x_i)$ with $i\in J$.
		In particular, we can find $N$ subfamilies $J_k\subset J$ such that
		$J = \cup_{k=1}^N J_k$ and the balls $\{B_{2R}(x_i)\}_{i\in J_k}$ are disjoint for every $k$.
		Therefore, there exists at least one family $I'=J_{\bar k}$ such that
		\[
		\mu (\bigcup_{i\in I'}B_{R}(x_i)) \geq \frac1N\mu(\R^n).
		\]
		The lemma is then proved with $\eps_0=\frac1N$.
	\end{proof}

\section{Whitney-type extension lemma}\label{a.Whitney}
In this subsection we prove the extension result which allows to truncate fields $A$ and to deduce an $L^\infty$ bound in Lemma \ref{l.Linfty}. This result is contained in \cite{LL2}. For the reader's convenience we give the details below.

\begin{lemma}\label{lemma-extension}
Let $W\subseteq\Omega_0$  be open, $A\in C(\Omega_0;\R^{2\times 2})$, and assume that there is an open set  $U\subseteq W$ such that:
\begin{itemize}
\item[(W1)] $B_{\eps}(x) \setminus W \neq \emptyset$ for every $x\in U$,
\item[(W2)] $\Curl A =0$ in the sense of distributions in $B_{6\eps}(U)$,
\item[(W3)] there exists $M>0$ such that for all $x\in B_{6\eps}(U)\setminus W$
\[
\sup_{r\in (0,\eps)} \mean_{B_r(x)} |A| \leq M.
\]
\end{itemize}
Then, there exist a dimensional constant $C>0$ and a field $\tilde A \in C(\Omega_{0}; \R^{2\times 2})$, such that
\begin{gather}
\Curl \tilde A= \Curl A \qquad \textup{in}\; \Omega_{0},\label{e.W-ext-1}\\
\big\{\tilde A \neq A \big\}\subset B_\eps(U),\label{e.W-ext-2}\\
|\tilde A(x)| \leq C M \qquad \forall\; x \in B_\eps(U).\label{e.W-ext-3}
\end{gather}
\end{lemma}

\begin{proof}
We fix a maximal set of points $\{p_1,\ldots, p_N\}\subset U$ satisfying $|p_i-p_j|\geq\eps/4$ for $i\neq j$. Hence,
\[
U\subset \bigcup_{\ell=1}^N B_{\frac{\eps}{2}}(p_\ell).
\]
Since $\Curl A=0$ in $B_{6\eps}(U)$, for each $\ell$ we can find a function $u_\ell\in C^1(B_{6\eps}(p_\ell))$ such that $\nabla u_\ell=A$ in $B_{6\eps}(p_\ell)$.
Consider the compact set
\[K=\bigcup_{\ell=1}^N \overline B_{6\eps}(p_\ell)\setminus W.\]
From (W3) we have that the maximal function of $A$ is bounded in $K$  by $M$, and therefore we deduce that $u_\ell\vert_{B_{6\eps}(p_\ell)\cap K}$ is Lipschitz continuous with constant $C M$, with $C>0$ a dimensional constant.

Consider next a Whitney decomposition of $W$, i.e., a family of dyadic closed cubes
with disjoint interior, $Q_j$ such that
\[W=\bigcup_j Q_j \qquad \textup{and}\qquad \frac{1}{3C} \dist(Q_j, W^c)\leq \diam(Q_j)\leq \frac{1}{C} \dist(Q_j, W^c), \qquad C>2.\]
If a cube $Q_j$ intersects $B_{\eps/2}(U)$, then $\dist(Q_j, W^c)<3\eps/2$ because of condition (W1). As a consequence, for any such cube there exists $z_j\in \partial W$ such that $|z_j - x_{Q_j}|=\dist(x_{Q_j}, W^c)$, where $x_{Q_j}$ is the center of the cube $Q_j$, and moreover by triangular inequality $|x_{Q_j}-z_j|<3\eps$.
Let $\{\varphi_j\}$ a partition of unity subordinated to the family $\{\hat Q_j\}$ where $\hat Q_j$ is the dilation of $Q_j$ with fixed center and factor $\frac54$, i.e., $\hat Q_j= x_{Q_j} + \frac54(Q_j-x_{Q_j})$.
By the property of the Whitney decomposition $\hat Q_j\subset W$.
Then for every $\ell$ we set
$$
\tilde u_\ell(x) :=
\begin{cases}
\sum_{j :\  \hat Q_j\cap B_{\eps/2}(p_\ell)\neq\emptyset}\varphi_j(x)[u_\ell(z_j)+ A(z_j)(x-z_j)]& \textrm{if }x\in B_{\eps/2}(p_\ell) \cap  U,\\
u_\ell(x) & \textrm{if }x\in B_{\eps/2}(p_\ell)\setminus U.
\end{cases}
$$
Note that, if $\hat Q_j\cap B_{\eps/2}(p_\ell)\neq\emptyset$, then
\begin{align*}
|x_{Q_j}-p_\ell|&\leq \frac{\eps}{2}+ \frac{\diam(\hat Q_j)}{2}
=  \frac{\eps}{2}+\frac{5\diam(Q_j)}{8}
\leq \frac{\eps}{2}+\frac{5 \dist(Q_j,K)}{8C}\\
&\leq \frac{\eps}{2}+\frac{15\eps}{8C}  < \eps.
\end{align*}
This implies that
\[
|z_{j}-p_\ell|\leq
|z_{j}-x_{Q_j}|+|x_{Q_j}-p_\ell| <4\eps.
\]
Thus, the definition of $\tilde u_\ell$ is well-posed because $u_\ell(z_j)$ is defined.
By Whitney Extension Theorem, $\tilde u_\ell$ is a $C^1$ extension of ${u_\ell}_{\vert_ {B_{\eps/2}(p_\ell)\setminus W}}$ in $ B_{\eps/2}(p_\ell)$.

If $x\in B_{\eps/2}(p_\ell)\cap  B_{\eps/2}(p_m)$, then there exists a constant $c\in \R$ such that $u_\ell=u_m+c$ in $ B_{6\eps}(p_\ell)\cap  B_{6\eps}(p_m)$ because $\nabla u_\ell = \nabla u_m=A$ in $ B_{6\eps}(p_\ell)\cap  B_{6\eps}(p_m)$. In this case it is simple to verify that also the gradients of the extensions coincide.
	Indeed, for every $x\in U$ let $\mathcal{I}_{x}$ be the set of indices of squares $Q_i$ such that  $\varphi_i(x) \neq 0$ (note that $\mathcal{I}_{x}$ is made at most of $4$ squares); then $\sup_{i\in \mathcal{I}_x} \nabla \varphi_i(x) = 0$ and
\begin{align*}
\nabla\tilde u_\ell(x)&
=\sum_{j :\  \hat Q_j\cap B_{\eps/2}(p_\ell)\neq\emptyset}\varphi_j(x) A(z_j)+\nabla \varphi_j(x) [u_\ell(z_j)+ A(z_j)(x-z_j)]\\
&=\sum_{j \in \mathcal{I}_x}\varphi_j(x) A(z_j)+\nabla \varphi_j(x) [u_m(z_j)+c+ A(z_j)(x-z_j)]\\
&=\sum_{j :\  Q_j\cap B_{\eps/2}(p_m)\neq\emptyset}\varphi_j(x) A(z_j)+\nabla \varphi_j(x) [u_m(z_j)+ A(z_j)(x-z_j)]\\
&=\nabla\tilde u_m(x).
\end{align*}
This implies that
\[
\tilde A(x) :=
\begin{cases}
\nabla \tilde u_\ell(x) & \textup{if }x\in B_{\eps/2}(p_\ell)\cap \tilde \Omega, \qquad \ell=1, \ldots, N, \\
A(x) & \textup{if }x \in \tilde \Omega\setminus \bigcup_{\ell}B_{\eps/2}(p_\ell),
\end{cases}
\]
is well-defined. Moreover, by construction $\Curl \tilde A = \Curl A = 0$ in $\bigcup_{\ell}B_{\eps/2}(p_\ell)\cap \tilde \Omega$, thus implying \eqref{e.W-ext-1} and \eqref{e.W-ext-2}.
Moreover, for every $x\in B_{\eps/2}(p_\ell)\cap U$ we have that
\begin{equation}
\begin{split}
|\nabla\tilde u_\ell(x)|\leq & \|A\|_{L^\infty( B_{\eps}(U)\setminus W)} +\left|\sum_{j :\  \hat Q_j\cap B_{\eps/2}(p_\ell)\neq\emptyset}\nabla \varphi_j(x) [u_\ell(z_j)+ A(z_j)(x-z_j)]\right|.
\end{split}
\end{equation}
For $i,j\in \mathcal{I}_x$ such that $\hat Q_j\cap \hat Q_i\neq \emptyset$ we have that
\begin{align*}
\left\vert u_\ell(z_j)+ A(z_j)(x-z_j) -
u_\ell(z_i) \right\vert &
\leq \left\vert u_\ell(z_j) -u_\ell(z_i) \right\vert
+ \left\vert A(z_j)(x-z_j)  \right\vert\\
&\leq CM |z_j-z_i| + C |A(z_j)| |x-z_j|
\\
&\leq C (M + \|A\|_{L^\infty(B_{\eps}(p_\ell)\setminus W)}) \, \diam(Q_j)\qquad \forall\:x\in \hat Q_j,
\end{align*}
where we used that
\begin{align*}
|z_j-z_i| & \leq |z_j-x_{Q_j}| + |x_{Q_j}-x_{Q_i}| +
|x_{Q_i}-z_i|\\
&\leq C \left(\diam(Q_j) + \diam(Q_i)\right) \leq C\diam(Q_j),
\end{align*}
and
\begin{align*}
|x-z_j| & \leq \diam(\hat Q_j)/2+|z_j-x_{Q_j}|
\leq C \diam(Q_j),\qquad \forall\:x\in \hat Q_j,
\end{align*}
since any two intersecting squares have comparable diameters.
Therefore, for every $x\in W\cap B_{\eps/2}(p_\ell)$, let $i\in \mathcal{I}_x$ be fixed and,
using that $\sum_{j\in \mathcal{I}_x} \nabla \varphi_j(x)=0$
because $\sum_{j\in \mathcal{I}_x} \varphi_j=1$ in a neighborhood of $x$,
we can estimate as follows:
\begin{align*}
&\left|\sum_{j \in \mathcal{I}_x}\nabla \varphi_j(x) [u_\ell(z_j)+ A(z_j)(x-z_j)]\right|\\
&\leq
\left|\sum_{j \in \mathcal{I}_x}\nabla \varphi_j(x) [u_\ell(z_j)+ A(z_j)(x-z_j) - u_\ell(z_i)]\right|\\
&\leq \sum_{j \in \mathcal{I}_x} \|\nabla \varphi_j(x)\|_\infty
C (M + \|A\|_{L^\infty(B_{\eps}(U)\setminus W)}) \diam(Q_j)\leq C(M + \|A\|_{L^\infty(B_{\eps}(U)\setminus W)}).
\end{align*}
Joining the two estimate above we conclude that
$$
\|\nabla \tilde u_\ell\|_{ L^\infty(B_{\eps/2}(p_\ell)\cap W)}\leq C (M + \|A\|_{L^\infty(B_{\eps}(U)\setminus W)})\leq CM.
$$
Therefore, since $u_\ell\vert_{B_{\eps/2}(p_\ell)\setminus W}$ is $CM$-Lipschitz, we conclude that the $C^1$ function $\tilde u_\ell$ is Lipschitz with constant $CM$ and therefore
\[
\|\tilde A\|_{L^\infty(B_{6\eps}(U)\cap \tilde\Omega)} \leq C M.
\]
\end{proof}

\end{document}